\documentclass[11pt]{article}
\usepackage{amssymb}
\usepackage{amsmath}
\usepackage{amsthm} 
\usepackage{times}
\usepackage[mathscr]{eucal}
\usepackage{mathrsfs}

\allowdisplaybreaks[4] 

\newtheorem{lemma}{Lemma}

\newtheorem{theorem}{Theorem}

\renewenvironment{proof}{ 
\noindent{\bf Proof.} \rm}{\penalty-20\null\hfill $\square$}

\numberwithin{equation}{section}

\newcommand{\C}{{\mathbb C}}

\newcommand{\N}{{\mathbb N}}

\newcommand{\R}{{\mathbb R}}

\newcommand{\Z}{{\mathbb Z}}

\newcommand{\bfe}{\mathbf{e}}
\newcommand{\bff}{\mathbf{f}}

\newcommand{\bfn}{\mathbf{n}}

\newcommand{\bfu}{\mathbf{u}}
\newcommand{\bfv}{\mathbf{v}}
\newcommand{\bfw}{\mathbf{w}}
\newcommand{\bfx}{\mathbf{x}}

\newcommand{\bfz}{\mathbf{z}}

\newcommand{\bfeta}{\mbox{\boldmath $\eta$}}
\newcommand{\bfzeta}{\mbox{\boldmath $\zeta$}}
\newcommand{\bfphi}{\mbox{\boldmath $\phi$}}

\newcommand{\bfpsi}{\mbox{\boldmath $\psi$}}

\newcommand{\bfomega}{\mbox{\boldmath $\omega$}}

\newcommand{\bfC}{\mathbf{C}}

\newcommand{\bfH}{\mathbf{H}}

\newcommand{\bfL}{\mathbf{L}}

\newcommand{\bfU}{\mathbf{U}}

\newcommand{\bfW}{\mathbf{W}}
\newcommand{\bfX}{\mathbf{X}}

\newcommand{\cL}{{\cal L}}
\newcommand{\cN}{{\cal N}}

\newcommand{\cT}{{\cal T}}

\newcommand{\calB}{\mathscr{B}}

\newcommand{\calL}{\mathscr{L}}

\newcommand{\rmd}{\mathrm{d}}
\newcommand{\rme}{\mathrm{e}}
\newcommand{\rmi}{\mathrm{i}}

\newcommand{\Spe}{\sigma}
\newcommand{\Spess}{\Spe_{\rm ess}}
\newcommand{\SpessEN}{\widetilde{\Spe}_{\rm ess}}
\newcommand{\Spep}{\Spe_{\rm p}}
\newcommand{\Spea}{\Spe_{\rm a}}
\newcommand{\Sper}{\Spe_{\rm r}}
\newcommand{\nul}{{\rm nul\br}}
\newcommand{\deff}{{\rm def\br}}
\newcommand{\ind}{{\rm ind\br}}
\newcommand{\Cs}{\bfC^{\infty}_{0,\sigma}(\Omega)}
\newcommand{\Ls}{\bfL^2_{\sigma}(\Omega)}
\newcommand{\Ps}{P_{\sigma}}
\newcommand{\Wns}{\bfW^{1,2}_{0,\sigma}(\Omega)}
\renewcommand{\Re}{{\rm Re}}

\newcommand{\bfvr}{\bfv_{\rm r}}
\newcommand{\bfvi}{\bfv_{\rm i}}
\newcommand{\Tone}{T_1}
\newcommand{\Tstar}{t^*}
\newcommand{\bfvstar}{\bfv^*}
\renewcommand{\div}{\mathrm{div}\,}
\newcommand{\bfzero}{\mathbf{0}}
\newcommand{\br}{\hbox to 0.7pt{}}
\newcommand{\K}{$\sqcap \hskip -6.7 pt \sqcup$}

\newcounter{constants}
\newcommand{\cn}[2]{ \addtocounter{constants}{1}
\newcounter{c#1#2}
\setcounter{c#1#2}{\value{constants}}}
\newcommand{\cc}[2]{c_{\arabic{c#1#2}}}

\begin{document}

\title{\LARGE \bf Nonlinear spectral instability of steady-state flow of a viscous liquid past a rotating obstacle}

\author{Giovanni P.~Galdi, \ Ji\v{r}\'{\i} Neustupa}

\date{}

\maketitle

\begin{abstract}
We show that a steady-state solution 
to the system of equations
of a Navier-Stokes flow past a rotating
body is nonlinealy unstable if the associated  linear operator $\cL$ has a part
of the spectrum in the half-plane $\{\lambda\in\C;\ \Re\,
\lambda>0\}$.
Our result does not follow from known methods, 
mainly
because the basic nonlinear operator is not bounded in the
same space in which the instability is studied. As an 
auxiliary result of independent interest, we also show that the uniform growth bound of
the $C_0$--semigroup $\rme^{\cL t}$ is equal to the spectral bound
of operator $\cL$. 
\end{abstract}

\vspace{1mm} \noindent
{\it AMS math.~classification (2010):} \ 35Q30, 35B35, 47D60,
70K50, 76D05.



\section{Introduction} \label{Se1}

{\bf \ref{Se1}.1. The initial--boundary valued problem.} \
Suppose  a compact body $\calB$ moves rigidly in an otherwise quiescent Navier-Stokes liquid translating and rotating about the $x_1$--axis
with a constant angular velocity $\omega$ and  a constant velocity
$u_{\infty}$.  In order to avoid that the region of flow be time-dependent, instead of referring the motion of the liquid to an inertial frame, it is convenient to describe it 
from a coordinate system 
attached to the body. In such a system, the relevant equations then take the following 
form
\begin{align}
\partial_t\bfu-
(\omega\br\bfe_1\times\bfx+u_{\infty}\bfe_1)\cdot\nabla\bfu+
\omega\br\bfe_1\times\bfu+\bfu\cdot\nabla\bfu\
&=\ -\nabla p+\nu\Delta\bfu+\bff, \label{1.1a} \\
\noalign{\vskip 2pt}
\div\bfu\ &=\ 0, \label{1.1b}
\end{align}
where $\bfu$, respectively $p$, are the transformed velocity,
respectively the pressure, $\bfe_1$ denotes the unit vector in
the direction of the $x_1$--axis, $\bff$ is the transformed
external body force and $\bfx$ is the transformed spatial
variable. The system (\ref{1.1a}), (\ref{1.1b})  is considered for $(\bfx,t)\in \Omega\times(0,\infty)$, with $\Omega$ a fixed domain. The no--slip boundary
condition for the velocity on the surface of the body
transforms to
\begin{equation}
\bfu(\bfx,t)\ =\ u_{\infty}\bfe_1+\omega\br\bfe_1\times\bfx \qquad
\mbox{for}\ \bfx\in\partial\Omega, \label{1.1c}
\end{equation}
and the information that the fluid is at rest in infinity leads to
the condition
\begin{equation}
\bfu(\bfx,t)\ \to\ \bfzero \qquad \mbox{for}\ |\bfx|\to\infty.
\label{1.1d}
\end{equation}
The details on the used transformation and the way one can
obtain equation (\ref{1.1a}) from the ``classical''
Navier--Stokes equation are described in many previous papers,
see e.g.~\cite{Hi2}, \cite{FaNe1} or \cite{GaNe1}.

The global in time existence of weak solutions to the problem
(\ref{1.1a})--(\ref{1.1d}), with a prescribed initial velocity
in $\bfL^2(\Omega)$, which is  divergence free (in the sense of
distributions) and such that its normal component coincides
with the normal component of $u_{\infty}\br\bfe_1+
\omega\br\bfe_1\times\bfx$ for $\bfx\in\partial\Omega$ (in a
certain sense of traces), was proven by Borchers in \cite{Bo}.
(It also follows from paper \cite{Ne6} on the weak solvability
of the Navier--Stokes equations in a domain with generally
moving boundaries.) The existence of strong solutions on a
``short'' time interval $(0,T)$, under the condition that
$u_{\infty}=0$, has been proven by Hishida \cite{Hi2}, Galdi,
Silvestre \cite{GaSi1} and Cumsille, Tucsnak \cite{CuTu}. The
latter formulate the
problem in an inertial frame, where the region of flow is time-dependent. They consider a body force $\bff$ locally square
integrable from $(0,\infty)$ to $\bfW^{1,\infty}(\R^3)$ and the
no--slip boundary condition for the velocity on
the body $\mathscr B$. Their main result, reformulated in terms of
the transformed velocity $\bfu$ satisfying
(\ref{1.1a})--(\ref{1.1d}), states that for a given $\bfu_0\in\Ls\cap\Wns$,
there exists $T>0$ and a unique solution $\bfu$ of the problem
(\ref{1.1a})--(\ref{1.1d}) such that $\bfu|_{t=0}=\bfu_0$ and
\begin{equation}
\begin{aligned}[2]
& \bfu\ \in\ L^2\bigl(0,T_0;\, \bfW^{2,2}(\Omega)\bigr)\cap
C\bigl([0,T_0];\, \bfW^{1,2}(\Omega)\bigr), \\
& \partial_t\bfu-(\omega\br\bfe_1\times\bfx+u_{\infty}\bfe_1)\cdot
\nabla\bfu+\bfomega\times\bfu\ \in\ L^2\bigl((0,T_0;\,
\bfL^2(\Omega)\bigr)
\end{aligned} \label{1.7}
\end{equation}
for every $T_0\in(0,T)$. Moreover, either $T=\infty$  or else the norm of $\bfu$ in $\bfW^{1,2}(\Omega)$
tends to infinity for $t\to T-$. Although a translational
motion of  $\calB$ is not considered in \cite{CuTu} (which
corresponds to $u_{\infty}=0$), the results can be extended to
the case $u_{\infty}\not=0$ by means of standard arguments
similar to the case $\omega=0$, $u_{\infty}\not=0$, studied
e.g.~in \cite{Hey3}. \par Galdi and Silvestre \cite{GaSi3} studied a
more general problem, i.e.~the motion of a rigid body in a
viscous incompressible fluid under the action of a given force
and torque. They considered the body force acting on the fluid
to be potential, which corresponds to $\bff=\bfzero$ in
equation (\ref{1.1a}), but, by the same method,  one to extend their result to a large class of non-zero
body forces $\bff$. In \cite{GaSi3}, in addition to  velocity and pressure fields, also the translational velocity and the angular
velocity of  the body become unknown.
However, the main findings can  be used also for the problem at hand, by prescribing  translational
and angular velocities, and then
calculating the force and torque needed for the body
performs the requested motion. Thus, considering the particular case where
translational velocity and 
angular velocity have the same direction $\bfe_1$, one obtains an
existence result for problem (\ref{1.1a})--(\ref{1.1d}) in the class (\ref{1.7}), 
entirely analogous to that following from \cite{CuTu}.  Description of, and some further comments on the
results from \cite{Hi2}, \cite{GaSi1} and \cite{CuTu} can also
be found in \cite{GaNe2}. 
\par
Finally, among the many other works studying general
qualitative properties of the problem
(\ref{1.1a})--(\ref{1.1d}), we wish to mention e.g.~the papers
\cite{DeKrNe1}, \cite{DeKrNe2} (by Deuring, Kra\v{c}mar and
Ne\v{c}asov\'a), \cite{Fa2} (by Farwig), \cite{FaKrNe} (by
Farwig, Krbec and Ne\v{c}asov\'a), \cite{Ga2}, \cite{Ga3}, (by
Galdi), \cite{GaSi2} (by Galdi and Silvestre), \cite{GeHeHi}
(by Geissert, Heck and Hieber), \cite{Hi3}, \cite{Hi1} (by
Hishida) and \cite{HiSh} (by Hishida and Shibata).

\vspace{4pt} \noindent
{\bf \ref{Se1}.2. Steady-state solution and  perturbation equations.} \ We further suppose that $\Omega$ is (an
exterior) domain in $\R^3$ with a $C^{2+\mu}$ boundary
$\partial\Omega$, for some $\mu\in(0,1)$,  and let $\bfU,\Pi$ be velocity and pressure field of a steady-state solution to
problem (\ref{1.1a})--(\ref{1.1d}) with the following properties: there
exists $r_0\in(1,3)$ such that
\begin{equation}
\nabla\bfU\in L^{r_0}(\Omega)^{3\times 3}\cap L^3(\Omega)^{3\times
3} \label{1.3}
\end{equation}
and there exist $\cn01\cc01,\, \cn02\cc02>0$ such that
\begin{equation}
|\bfU(\bfx)|\leq\cc01, \qquad
|\nabla\bfU(\bfx)|\leq\frac{\cc02}{1+|\bfx|} \qquad \mbox{for}\
\bfx\in\Omega. \label{1.4}
\end{equation}
The existence of a steady-state solution with these properties for a
large class of body forces $\bff$ follows e.g.~from
\cite[Sec.~XI.6]{Ga1} (for $u_{\infty}\not=0$) and
\cite[Sec.~XI.7]{Ga1} or \cite{Ga3} (for $u_{\infty}=0$ and
$\omega$ ``sufficiently small'').

As we are mainly interested in instability of solution $(\bfU,\Pi$, it is useful to write the solutions of
(\ref{1.1a})--(\ref{1.1d}) in the form $\bfu=\bfU+\bfv$,
$p=\Pi+q$, where the perturbations $\bfv$ and $q$ satisfy the
equations
\begin{align}
\partial_t\bfv-(\omega\br\bfe_1\times\bfx+u_{\infty}\bfe_1)\cdot
\nabla\bfv & +\omega\br\bfe_1\times\bfv+\bfU\cdot\nabla\bfv+
\bfv\cdot\nabla\bfU+\bfv\cdot\nabla\bfv \nonumber \\
&=\ -\nabla q+\Delta\bfv, \label{1.6a} \\ \noalign{\vskip 2pt}
\div\bfv\ &=\ 0 \label{1.6b}
\end{align}
in $\Omega\times(0,\infty)$, and the conditions
\begin{alignat}{5}
\bfv(\bfx,t)\ &=\ && \bfzero \hspace{12mm} &&
\mbox{for}\ \bfx\in\partial\Omega, \label{1.6c} \\
\noalign{\vskip 2pt}
\bfv(\bfx,t)\ &\to\ && \bfzero && \mbox{for}\ |\bfx|\to\infty.
\label{1.6d}
\end{alignat}

\vspace{4pt} \noindent
{\bf \ref{Se1}.3. On steady-state stability.} \ Although
stability is not the main subject of this paper, we would like to recall some
corresponding relevant works. Most
of them concern the case $\omega=0$, which describes 
the situation when $\calB$ does not rotate and just
moves  with the translational velocity
$u_{\infty}\bfe_1$. The time--decay of perturbations of
solution $(\bfU,\Pi)$ in appropriate norms and under various
conditions of smallness imposed on $\bfU$, was proven in the
works \cite{Hey2}, \cite{Hey3} (by Heywood), \cite{Mas} (by
Masuda) and further on in a series of other papers, see
e.g.~\cite{GaNe2} for a detailed list of corresponding
references. The assumption of ``sufficient smallness'' on
$\bfU$ is avoided in the papers \cite{Ne3},
\cite{DeNe}, \cite{De}, \cite{Ne5} and \cite{Saz}. It is
worth remarking that, in these last four papers, the stability of
 $(\bfU,\Pi)$ is shown to be determined just by the location 
of the eigenvalues of the associated relevant linear operator $\cL$,
disregarding the presence of an essential spectrum $\Spess(\cL)$, which is
non-empty and touches the imaginary axis from the left at point
$0$. The stability of a steady-state solution $(\bfU,\Pi)$ in the
case when $\omega\not=0$, under some conditions of smallness of
$\bfU$, was proved in the papers \cite{GaSi1}, \cite{HiSh} (in the
case $u_{\infty}=0$) and \cite{Sh3}. Sufficient conditions for
stability of solution $\bfU$ without the condition of smallness
of $\bfU$ have been derived in \cite{GaNe1}. Here, in analogy
with \cite{Ne3}, the authors use the assumption of an
appropriate time--decay property of the semigroup $\rme^{\cL t}$,
applied to a finite family of certain functions.

\vspace{4pt} \noindent
{\bf \ref{Se1}.4. Brief overview of available methods on spectral
instability.} \ In view of their intrinsic pertinence to our main finding,
we wish to recall some of the most relevant and currently available
methods on instability of steady solutions to various
evolutionary differential equations.

The classical result of Coddington and Levinson \cite{CoLe}
concerns the equation
\begin{equation}
\frac{\rmd\bfx}{\rmd t}\ =\ \cL\bfx+\cN(\bfx,t) \label{1.5}
\end{equation}
in $\R^n$, where $\cL$ is a real $n\times n$ matrix and
$|\cN(\bfx,t)|=o(|\bfx|)$ for $|\bfx|\to 0$ uniformly with
respect to $t\in(0,\infty)$. Theorem XIII.1.2 in \cite{CoLe}
says that the zero solution $\bfx(t)\equiv\bfzero$ is unstable
under the assumption that matrix $\cL$ has at least one
eigenvalue with positive real part. This classical result was successively generalized by different authors to systems of PDE's, by regarding (\ref{1.5}) as an abstract evolution 
equation in Hilbert or Banach spaces; see  the books \cite{DaKr}, \cite{He} and the
papers \cite{Ki1}, \cite{ShaStr}.  In these works it is assumed that $\cL$ and $\cN$ are  appropriate linear and
nonlinear operators, respectively, and that the spectrum of $\cL$ has an nonempty
intersection with the half plane $\{\lambda\in\C;\ \Re\,
\lambda>0\}$. Concerning further properties requested on $\cL$, we begin to recall that Daleckij and
Krejn \cite{DaKr} considered equation (\ref{1.5}) in a Banach
space $\bfX$ and assumed  $\cL$   bounded and closed
 in $\bfX$. (See \cite[Theorem VII.2.3]{DaKr}.)
Kielh\"ofer \cite{Ki1} studied  (\ref{1.5}) as an
equation in a Hilbert space $\bfH$ and assumed that $\cL=A+M$,
where $A$ is a (linear) selfadjoint and positive definite
 with compact inverse, while $M$ satisfies $D(M)\supset D(A^{\beta})$ for
some $\beta\in[0,\frac{1}{2})$ and $\|M\bfu\|\leq c\,
\|A^{\beta}\bfu\|$ for $\bfu\in D(A^{\beta})$. Henry \cite{He}
considered  (\ref{1.5}) as an equation in a Banach
space $\bfX$ and assumed that $\cL$ is a sectorial operator in
$\bfX$. (See Theorem 5.3.1 in \cite{He}.) It should be remarked that, directly or indirectly, in all works \cite{DaKr,Ki1,He} the operator $\cL$ is supposed to be the generator of an analytic semigroup.   Particularly significant, in this sense, becomes then the contribution by Shatah and Strauss 
\cite{ShaStr} who only require $\cL$ to be the generator of a $C_0$--semigroup in the Banach space $\bfX$ where (\ref{1.5}) is studied.
Concerning the assumptions on the operator $\cN$, it must be emphasized that in all the mentioned works \cite{DaKr},
\cite{Ki1}, \cite{He} and \cite{ShaStr}, it is supposed
that, in a neighborhood of 0,  $\cN(\bfx,t)$ is ``suficiently small" compared to $\bfx$
 in the  norm of the space $\bf X$, with respect to which 
instability is investigated, a condition that is {\em not} met by the problem studied in this paper.
\par
In addition to these general results, we would like to mention also those proved directly for Navier-Stokes equations. The problem
(\ref{1.1a})--(\ref{1.1d}) with $\omega=u_\infty=0$ was studied in a
bounded domain $\Omega\subset\R^3$ by Sattinger \cite{Sat}. The
author reformulated the question of stability and 
instability of a steady-state solution as the same question
concerning the zero solution of an equation of the type
(\ref{1.5}) in $\Ls$ (see subsection 2.1 for the definition of
this function space). Moreover, he showed that the operator $\cL$ has a
compact resolvent and proved, in particular,  instability in the norm of $\Ls$ under the assumption that
some of the eigenvalues of $\cL$ have positive real part. An
analogous result can be found in the book \cite{Yu} by Yudovich
and in the paper \cite{FriStrVi} by Friedlander et al., who
proved the instability of a steady flow in an $n$--dimensional
finite domain in the norm of $\bfH^s(\Omega)$ for
$s>\frac{1}{2}n+1$. Sazonov \cite{Saz} treated  problem
 (\ref{1.1a})--(\ref{1.1d}) with $\omega=0$ in an
exterior domain $\Omega\subset\R^3$. He proved the instability
of a steady-state flow in the norm of $\bfL^3(\Omega)$, assuming that
at least one eigenvalue of the associated linearized operator
operator $\cL$ has positive real part and that $\bf U(x)$ has suitable summability properties for large $|x|$. This question has
been reconsidered independently by Friedlander {\em et al.}~in \cite{FriPaSh},
where the authors deal with a steady flow $\bfU$ in a bounded or
unbounded domain $\Omega$ of $\R^n$, assuming that $\bfU$ has
derivatives of all orders bounded in $\Omega$ and the
associated linear operator $\cL$ has a part of the spectrum in
the right half-plane of $\C$. Instability is proven in the
norm of $\bfL^r(\Omega)$ for any $r\in(1,\infty)$. We wish to emphasize that for the proof of all the above results it is {\em crucial} that the operator $\cL$  be  the generator of an {\em analytic} semigroup. 

\vspace{4pt} \noindent
{\bf \ref{Se1}.5. Main results of this paper.} \ We treat 
problem (\ref{1.1a})--(\ref{1.1d}) in an exterior domain
$\Omega$, and, in contrast to \cite{Saz} and \cite{FriPaSh}, we
consider the case $\omega\not=0$. This implies that the semigroup
$\rme^{\cL t}$ generated by the associated linear operator
$\cL$ is {\em no longer analytic} but only of  class $C_0$ (see
subsection 2.5), which forces us to employ  different
estimates and technique than those, for example, of \cite{Saz} or \cite{FriPaSh}.
We prove the instability of solution $(\bfU,\Pi)$ in the norms of
$\bfL^2(\Omega)$ and $\bfW^{1,2}(\Omega)$, under the assumption
that $\cL$ has at least one eigenvalue in the
half-plane $\{\lambda\in\C;\ \Re\, \lambda>0\}$. (Although the
operator $\cL$ has a non--empty essential spectrum, we show
that the right half-plane in $\C$ may contain only eigenvalues
of $\cL$ and the number of eigenvalues with real parts
exceeding any given $\xi>0$ is finite, see subsection
\ref{Se2}.8.) It is necessary to stress that none of the abstract general results
from \cite{DaKr}, \cite{Ki1}, \cite{He} or \cite{ShaStr} can be
directly applied to our problem. In addition to the fact that
our operator $\cL$ does not satisfy the assumptions from
\cite{DaKr}, \cite{Ki1} or \cite{He}, the main reason is that
the nonlinear operator $\cN$ is not bounded in the same space
in which the stability or instability is considered. (See
subsection \ref{Se2}.2.) The statement on the instability is
formulated in Theorem \ref{T2}.

As an important auxiliary result, we also prove that the uniform growth
bound $\gamma(\rme^{\cL t})$ of the $C_0$--semigroup $\rme^{\cL
t}$, as a semigroup in $\Ls$, is equal to the spectral bound
$s(\cL)$ of operator $\cL$. We wish to point out that, as is well known,  such an
equality does not generally hold for $C_0$--semigroups and the
question under which additional conditions the equality holds
belongs to the most interesting and challenging problems in the
theory of the $C_0$--semigroups. The result is formulated in
Theorem~\ref{T1}.

\section{The associated linear and nonlinear operators and the
operator form of (\ref{1.6a})--(\ref{1.6d})} \label{Se2}

\noindent
{\bf \ref{Se2}.1. Notation.} \ We denote vector functions and
spaces of vector functions by boldface letters. We also denote
by $c$ a generic constant whose value may change from line
to line. On the other hand, $c$'s with indices denote constants
with fixed values.

\begin{list}{$\circ$}
{\setlength{\topsep 0pt}
\setlength{\itemsep 0pt}
\setlength{\leftmargin 14pt}
\setlength{\rightmargin 0pt}
\setlength{\labelwidth 10pt}}

\item
For $R>0$, we set $\Omega_R:=\Omega\cap B_R(\bfzero)$ and
$\Omega^R:=\Omega\cap \{\bfx\in\R^3;\ |\bfx|>R\}$.

\item
For $1<r\leq\infty$ and $k\in\N$, we denote by $\|\, .\, \|_r$ or
$\|\, .\, \|_{k,r}$ the norm of  scalar-- , vector-- or 
tensor--valued function with components in $L^r(\Omega)$ or
$W^{k,r}(\Omega)$, respectively. If the norm is considered on
another domain than $\Omega$ then we use. e.g.,~the notation $\|\,
.\, \|_{r;\, \Omega_R}$, etc.

\item
$|\, .\, |_{1,r}:=\|\nabla\br.\, \|_r$ and $|\, .\,
|_{2,r}:=\|\nabla^2\br.\, \|_r$,

\item
$(\, .\, ,\, .\, )_2$ is the scalar product in $\bfL^2(\Omega)$.

\item
$\Cs$ denotes the space of infinitely differentiable
divergence--free vector functions with a compact support in
$\Omega$. For $1<r<\infty$ and $k\in\{0\}\cup\N$, we denote by
$\bfL^r_{\sigma}(\Omega)$, respectively
$\bfW^{k,r}_{0,\sigma}(\Omega)$, the closure of $\Cs$ in
$\bfL^r(\Omega)$, respectively in $\bfW^{k,r}_0(\Omega)$.

\item
The orthogonal projection of $\bfL^2(\Omega)$ onto $\Ls$ is
denoted by $\Ps$.

\item
$\calL(\Ls)$ denotes the space of bounded linear operators in
$\Ls$.

\item
We set $A\bfphi:=\Ps\Delta\bfphi$ for $\bfphi\in D(A):=
\Wns\cap\bfW^{2,2}(\Omega)$. The operator $A$ (Stokes
operator) is non--positive and selfadjoint in $\Ls$. Its
domain $D(A)$ is a Banach space with the norm $\|A\, .\, \|_2+\|\,
.\, \|_2$. Further, for $\bfphi\in D(A)$ we define
\begin{align*}
B^0\bfphi\ &:=\ (\bfe_1\times\bfx)\cdot\nabla
\bfphi-\bfe_1\times\bfphi, \\ \noalign{\vskip 1pt}
B^1\bfphi\ &:=\ \partial_1\bfphi, \\ \noalign{\vskip 1pt}
B^2\bfphi\ &:=\ -\Ps\bigl[\bfU\cdot\nabla\bfphi+
\bfphi\cdot\nabla\bfU\bigr]
\end{align*}
By using (\ref{1.3}) and (\ref{1.4}), one
can easily verify that the range of $B^2$ is  in $\Ls$. In fact, 
 in \cite{FaNe2} or \cite{Ne3} it is shown that $B^1\bfphi$ is
 in $\bfW^{1,2}(\Omega)\cap\Ls$ for $\bfphi\in
D(A)$, hence $B^1$ can also be considered to be an operator in
$\Ls$. Also, it is clear that $B^0\bfphi\in \bfW^{1,2}_{loc}(\Omega)\cap L^2_\sigma(\Omega)$. 
\par
We next define $$D(\cL^0)\equiv
D(\cL):=\{\bfphi\in D(A);\ (\omega\br\bfe_1\times
\bfx)\cdot\nabla\bfphi\in\bfL^2(\Omega)\}\,,$$ with operators $\mathcal L_0$ and $\mathcal L$ given by
\begin{alignat*}{3}
& \cL^0\bfphi\ &&:=\ \nu A\bfphi+\omega\br B^0\bfphi+u_{\infty}\br
B^1\bfphi, \\ \noalign{\vskip 1pt}
& \cL\bfphi\ &&:=\ \nu A\bfphi+\omega\br B^0\bfphi+u_{\infty}\br
B^1\bfphi+B^2\bfphi.
\end{alignat*}
The symmetric part $\cL_s$ and the skew--symmetric part (=
anti-symmetric part) $\cL_a$ of the operator $\cL$ are
\vspace{-8pt}
\begin{displaymath}
\cL_s=\nu A+B^2_s, \qquad \cL_a=\omega\br B^0+u_{\infty}\br
B^1+B^2_a,
\end{displaymath}
where
\vspace{-8pt}
\begin{align*}
B^2_s\bfphi\ &=\ -\Ps\bigl[\bfphi\cdot(\nabla\bfU)_s\bigr], \\
\noalign{\vskip 2pt}
B^2_a\bfphi\ &=\ -\Ps\bigl[\bfU\cdot\nabla\bfphi+\bfphi\cdot
(\nabla\bfU)_a\bigr].
\end{align*}

\item
Finally, we denote by $\cN$ the operator defined as follows:
\vspace{-5pt}
\begin{displaymath}
\cN\bfphi\ :=\ -\Ps(\bfphi\cdot\nabla\bfphi) \qquad (\mbox{for}\
\bfphi\in D(A)).
\end{displaymath}

\end{list}

\vspace{4pt} \noindent
{\bf \ref{Se2}.2. Important inequalities.} \ Conditions
(\ref{1.3}), (\ref{1.4}) and Sobolev's inequality (see
e.g.~\cite[p.~54]{Ga1}) imply that $\bfU\in\bfL^a(\Omega)$ for
all $3r_0/(3-r_0)\leq a<\infty$.
By using the latter along with (\ref{1.3}), (\ref{1.4}), and  H\"older and Sobolev
inequalities, one shows that  $B^1$, $B^2$ and $\cN$ satisfy the following bounds for all $\bfphi\in D(A)$ $\cn62\cn63$
\begin{alignat}{3}
& \| B^1\bfphi\|_2\ &&\leq\ |\bfphi|_{1,2}, \label{2.5a} \\
\noalign{\vskip 2pt}
& \| B^2_s\bfphi \|_2+\| B^2_a\bfphi \|_2\ &&\leq\ \cc62\,
|\bfphi|_{1,2}, \label{2.5b} \\ \noalign{\vskip 2pt}
& \|\cN\bfphi\|_2\ &&\leq\ \cc63\, \|A\bfphi\|_2^{1/2}\,
|\bfphi|_{1,2}^{3/2}. \label{2.5f}
\end{alignat}
It is proven in
\cite{GaNe1} that the operator $B^2$ is relatively
$A$--compact, relatively $(\nu A+\omega\br B^0)$--compact and
relatively $\cL^0$--compact in the space $\Ls$.

\vspace{4pt} \noindent
{\bf \ref{Se2}.3. An operator form of equations (\ref{1.6a}),
(\ref{1.6b}).} \ Applying standard arguments, one can show that
the system of equations (\ref{1.6a}), (\ref{1.6b}) can be
written as an operator equation
\begin{equation}
\frac{\rmd\bfv}{\rmd t}\ =\ \cL\bfv+\cN\bfv \label{2.2}
\end{equation}
in the space $\Ls$. Here and further on, we mostly consider $\bfv$
to be a function of one variable $t$ with values in an appropriate
function space. (This justifies writing the derivative with
respect to time as $\rmd\bfv/\rmd t$ and not $\partial_t\bfv$.)

From the results of papers \cite{GaSi3} and \cite{CuTu} one deduces that for a given $\bfv_0\in\Wns$ there exists $T\in(0,\infty]$ and a solution $\bfv$
of equation (\ref{2.2}) on the time interval $(0,T)$ such that
$\bfv(0)=\bfv_0$ and
\begin{equation}
\begin{aligned}[2] & \bfv\in L^2\bigl(0,T_0;\, D(A)\bigr)\cap
C\bigl([0,T_0);\, \Wns\bigr), \\ & \frac{\rmd\bfv}{\rmd\br
t}-\omega\br B^0\bfv-u_{\infty}\br B^1\bfv\in L^2\bigl(0,T_0;\,
\Ls\bigr) \end{aligned} \label{2.3}
\end{equation}
for each $T_0\in[0,T)$. Moreover, if $T<\infty$ then
$\|\bfv(t)\|_{1,2}\to\infty$ for $t\to T-$. Note that such a
solution is not classical, but it is more than just a mild
solution. The main reason is that,  due to (\ref{2.3}),
equation (\ref{2.2}) makes sense, as an equation in $\Ls$, for
a.a.~$t\in(0,T)$.

\vspace{4pt} \noindent
{\bf \ref{Se2}.4. Spectra of the operators $\cL^0$ and $\cL$.}
\ Recall that a closed densely defined operator $S$ in a Banach
space $\bfX$ is said to be {\it Fredholm} if its range is
closed and both $\nul(S)$ (the nullity of $S$) and $\deff(S)$
(the defficiency of $S$) are finite. Operator $S$ is called
{\it semi-Fredholm} if its range is closed and at least one of
the numbers $\nul(S)$, $\deff(S)$ is finite. According to the
deffinition from \cite[p.~243]{Ka}, the {\it essential
spectrum} of $S$ is the set of those $\lambda\in\C$, for which
the operator $S-\lambda I$ is not semi-Fredholm. We denote by
$\Spe(S)$ the whole spectrum, by $\Spep(S)$ the point spectrum
and by $\Spess(S)$ the essential spectrum of $S$. Note that one
can also find a different definition of the essential spectrum
in the literature, according to which $\lambda$ belongs to the
essential spectrum of $S$ if the operator $S-\lambda I$ is not
Fredholm. (See e.g.~the footnote on p.~243 in \cite{Ka} or
\cite{Sch1}.) In order to distinguish between the two
definitions, we denote by $\SpessEN(S)$ the essential spectrum
of operator $S$, satisfying the second definition. (We shall
also denote by tilde some other quantities, related to the
second definition of the essential spectrum.) Obviously,
$\Spess(S)\subset\SpessEN(S)$. Recall that while $\Spess(S)$ is
preserved under relatively compact additive perturbations of
$S$, see \cite[Theorem IV.5.35]{Ka}, $\SpessEN(S)$ is preserved
under compact additive perturbations of $S$, see
\cite[p.~248]{EnNa} or \cite[Corollary 2.2]{Sch1}.

The next two formulas provide a characterization of the essential
spectrum of operator $\cL^0$; see \cite{FaNe1,FaNe2}:
\begin{equation}
\begin{aligned}
\Spess(\cL^0)&=\bigl\{\lambda=\alpha+\rmi\br k\omega\in\C;\
\alpha\leq 0,\ k\in\Z\bigr\} && \mbox{if}\ u_{\infty}=0, \\
\noalign{\vskip 0pt}
\Spess(\cL^0)&=\Bigl\{\lambda=\alpha+\rmi\beta+\rmi\br
k\omega\in\C;\ \alpha,\beta\in\R,\ k\in\Z,\
\alpha\leq-\frac{\nu\beta^2} {u_{\infty}^2}\Bigr\} \hspace{10pt}
&& \mbox{if}\ u_{\infty}\not=0.
\end{aligned} \label{2.10}
\end{equation}
We observe that if $u_{\infty}=0$ then $\Spess(\cL^0)$ is a
union of infinitely many equidistant half-lines parallel to the
real axis. If $u_{\infty}\not=0$ then $\Spess(\cL^0)$ consists
of a union of equally shifted filled in parabolas in the left
half-plane of $\C$. Similar results were obtained in
\cite{FaNe3} and \cite{FaNeNe} in the general $L^q$--setting.
It follows from \cite[Theorem 1.2]{FaNe3} that all
$\lambda\in\Spe(\cL^0)\smallsetminus\Spess(\cL^0)$ are isolated
eigenvalues of $\cL$ with negative real parts and finite
algebraic multiplicities. (This set may also be empty.) Thus,
since $\ind(\cL^0-\lambda I)\equiv \nul(\cL^0-\lambda
I)-\deff(\cL^0-\lambda I)$ is constant in
$\C\smallsetminus\Spess(\cL^0)$ (by \cite[Theorem
IV.5.17]{Ka}), we deduce that $\nul(\cL^0-\lambda
I)=\deff(\cL^0-\lambda I)<\infty$ for all
$\lambda\in\C\smallsetminus\Spess(\cL^0)$. Thus, $\cL^0-\lambda
I$ is a Fredholm operator for all
$\lambda\in\C\smallsetminus\Spess(\cL^0)$. Consequently,
$\Spess(\cL^0)=\SpessEN(\cL^0)$.

As the operators $\cL$ and $\cL^0$ differ only by operator
$B^2$, which is relatively $\cL^0$--compact, we have
$\Spess(\cL)=\Spess(\cL^0)$ (by \cite[Theorem IV.5.35]{Ka}).
Moreover, the operator $\cL-\alpha I$ is dissipative for large
positive $\alpha$, which can be easily verified by means of
estimate (\ref{2.5b}). Hence all $\lambda\in\C$ with positive
and sufficiently large real parts belong to the resolvent set
of $\cL$. Since the open set $\C\smallsetminus\Spess(\cL)$ has
just one component, we deduce by means of the same arguments as
in the previous paragraph that the set
$\Spe(\cL)\smallsetminus\Spess(\cL)$, if it is not empty,
consists of an at most countable family of isolated eigenvalues
of $\cL$ with finite algebraic multiplicities (which can
possibly cluster only at points of $\Spess(\cL)$) and
$\nul(\cL-\lambda I)=\deff(\cL-\lambda I)<\infty$ for all
$\lambda\in\C\smallsetminus\Spess(\cL)$. Hence $\cL-\lambda I$
is a Fredholm operator for all $\lambda\in\C\smallsetminus
\Spess(\cL)$ and $\Spess(\cL)=\SpessEN(\cL)$.

Thus, we observe that the operators $\cL^0$ and $\cL$ have the
same essential spectra described by (\ref{2.10}), no matter
which one of the two aforementioned definitions of the essential
spectrum we use. Moreover, recall that
$\Spe(\cL^0)\smallsetminus\Spess(\cL^0)$ and
$\Spe(\cL)\smallsetminus\Spess(\cL)$ can also contain isolated
eigenvalues with finite algebraic multiplicities, which may
cluster in $\C$ only at the boundary of the essential spectrum.
While all such eigenvalues of $\cL^0$ have negative real parts,
the eigenvalues of $\cL$ may also lie in the half-plane
$\{\lambda\in\C;\; \Re\, \lambda>0\}$.

\vspace{4pt} \noindent
{\bf \ref{Se2}.5. Semigroups generated by the operators
$\cL^0$ and $\cL$.} \ If $\omega=0$ then operator $\cL^0$
generates an analytic semigroup $\rme^{\cL^0t}$ in $\Ls$  for $u_{\infty}=0$ (e.g.~\cite{Mi1}), and  for
$u_{\infty}\not=0$ \cite{Ne3}. However, if $\omega\not=0$ (which we
assume) then the same operator generates only a
$C_0$--semigroup in $\Ls$, see \cite{Hi2} or \cite{GeHeHi} for
$u_{\infty}=0$ and \cite{Sh2} for $u_{\infty}\in\R$. As showed in
 \cite{GaNe1},  the operator $\cL$ generates a
$C_0$--semigroup also in $\Ls$. (This is obtained by
a relatively easy application of the Lumer--Phillips theorem.)
Both semigroups $\rme^{\cL^0t}$ and $\rme^{\cL t}$ are
strongly continuous, but none of them is eventually uniformly
continuous (i.e.~continuously dependent on $t$ in the topology
of $\calL\bigl(\Ls\bigr)$ in the interval $(t_0,\infty)$ for
some $t_0>0$). In this regard, we recall that the necessary condition for the
eventual uniform continuity is that the intersection of the
spectrum of the generator with each half-plane
$\{\lambda\in\C;\ \Re\, \lambda>b\}$ (for any $b\in\R$) is
bounded, see \cite[Theorem 4.4.18]{EnNa}. Here, however, one
can see from the shapes of $\Spe(\cL^0)$ and $\Spe(\cL)$ that
this necessary condition is not fulfilled.

\vspace{4pt} \noindent
{\bf \ref{Se2}.6. Some estimates of the semigroup
$\rme^{\cL^0t}$.} \ Although we need just the $L^2$--$L^2$
estimates of $\rme^{\cL^0t}$ and $\nabla\rme^{\cL^0t}$ in this
paper, we recall, for completeness, the more general
$L^r$--$L^s$ estimates, satisfied for
$\bfphi\in\bfL^s_{\sigma}(\Omega)$ and $t>0$: $\cn64\cn65$
\begin{alignat}{5}
& \|\rme^{\cL^0t}\bfphi\|_r\ && \leq\ \cc64\, t^{\frac{3}{2}\,
\left( \frac{1}{r}-\frac{1}{s} \right)}\, \|\bfphi\|_s &&
\mbox{for}\ 1<s\leq r<\infty, \label{2.11} \\ \noalign{\vskip 2pt}
& |\rme^{\cL^0t}\bfphi|_{1,r}\ && \leq\ \cc65\,
t^{-\frac{1}{2}+\frac{3}{2}\, \left( \frac{1}{r}-\frac{1}{s}
\right)}\, \|\bfphi\|_s \hspace{15pt} && \mbox{for}\ 1<s\leq r\leq
3. \label{2.12}
\end{alignat}
These inequalities are proved in \cite{HiSh}  for the case $u_{\infty}=0$, 
and in \cite{Sh2}  for the case $u_{\infty}\not=0$. Moreover, one can
also deduce from \cite[Theorem 1.1]{Sh4} that there exists
$\rho>0$ such that $\cn66$
\begin{equation}
|\rme^{\cL^0t}\bfphi|_{2,2}\ \leq\ \frac{\cc66}{t}\ \rme^{\rho t}\
\|\bfphi\|_2 \label{2.12a}
\end{equation}
for $\bfphi\in\Ls$ and $t>0$.

\vspace{4pt} \noindent
{\bf \ref{Se2}.7. The uniform growth bound of a general
$C_0$--semigroup.} \ Assume, for a while, that $\cT=\cT(t)$ is
a general $C_0$--semigroup in a Banach space $\bfX$ (with the
norm $\|\, .\, \|$), generated by operator $S$. The {\it
spectral bound} $s(S)$ of $S$ and the {\it uniform
growth bound} $\gamma(\cT)$ of the semigroup $\cT$ are defined
by the formulas
\begin{align*}
s(S)\ &:=\ \sup\, \bigl\{ \Re\, \lambda;\ \lambda\in\Spe(S)
\bigr\}, \\ \noalign{\vskip 2pt}
\gamma(\cT)\ &:=\ \inf\, \bigl\{\mu\in\R;\ \exists\br M_{\mu}>0\ \
\forall\br t>0\ \ \forall\br\bfx\in\bfX\ :\ \|\cT(t)\bfx\|\leq
M_{\mu}\, \rme^{\mu t}\, \|\bfx\| \bigr\},
\end{align*}
respectively. It is known that generally $s(S)\leq\gamma(\cT)$.
The question of ``under which conditions the equality
$s(S)=\gamma(\cT)$ holds'' is a classical prolem in
the theory of $C_0$--semigroups. It is known that $s(S)=
\gamma(\cT)$ if the semigroup $\cT$ satisfies the spectral
mapping theorem, i.e.~if $\Spe(\cT(t))\smallsetminus\{0\}=
\exp\bigl(t\, \Spe(S)\bigr)$ holds for some $t>0$. (See
\cite[Proposition 1.2.2]{JvN} or \cite[Proposition
2.2.2]{EnNa}.) However, while the identities
$\Spep(\cT(t))\smallsetminus \{0\}=\exp\bigl(t\,
\Spep(S)\bigr)$ and
$\Sper(\cT(t))\smallsetminus\{0\}=\exp\bigl(t\, \Sper(S)\bigr)$
hold for the point spectrum and the residual spectrum (see
Theorems 2.1.2 and 2.1.3 in \cite{JvN} or \cite[Theorem
IV.3.7]{EnNa}), the approximate point spectrum generally
satisfies only the inclusion $\exp\bigl(t\,
\Spea(S)\bigr)\subset\Spea(\cT(t))\smallsetminus\{0\}$. The
identity $\exp\bigl(t\, \Spea(S)\bigr)=\Spea(\cT(t))
\smallsetminus\{0\}$ (which implies the validity of the
spectral mapping theorem and consequently also the equality
$s(S)=\gamma(\cT)$), is satisfied e.g.~if the semigroup $\cT$
is eventually uniformly continuous and, in particular,  if it
is analytic. (See \cite[Theorem IV.3.10]{EnNa} or \cite[Theorem
2.3.2]{JvN}.)

Further in this subsection, we present some useful notions and
assertions, which are described and discussed in greater detail
e.g.~in \cite[pp.~249--258]{EnNa}. The {\it essential spectral
radius} of the semigroup $\cT$ at time $t$ is defined by the
formula
\begin{displaymath}
\widetilde{r}_{\rm ess}(\cT(t))\ :=\ \sup\, \bigl\{ |\lambda|;\
\lambda\in\SpessEN(\cT(t))\bigr\}.
\end{displaymath}
The essential spectral radius of $\cT(t)$ can also be
characterized as the infimum of the set of $\rho>0$ such that the
implication ``$\zeta\in\Spe(\cT(t)),\ |\zeta|>\rho\
\Longrightarrow\ \zeta$ is an eigenvalue of $\cT(t)$ with a finite
algebraic multiplicity" holds. The {\it essential growth bound} of
the semigroup $\cT$ is defined to be
\begin{displaymath}
\widetilde{\gamma}_{\rm ess}(\cT)\ :=\ \frac{1}{t_0}\ \ln\,
\widetilde{r}_{\rm ess}(\cT(t_0)) \qquad (\mbox{for any}\ t_0>0).
\end{displaymath}
(One can verify that the right hand side is independent of $t_0$
for $t_0>0$.) The uniform growth bound of the semigroup $\cT$
satisfies the important formula $\gamma(\cT)=\max\{
\widetilde{\gamma}_{\rm ess}(\cT);\ s(S) \}$.

\vspace{4pt} \noindent
{\bf \ref{Se2}.8. The uniform growth bounds of the semigroups
$\rme^{\cL^0t}$ and $\rme^{\cL t}$.} \ Our main objective in
this subsection is to prove the identities
$\gamma(\rme^{\cL^0t})= s(\cL^0)=0$ and $\gamma(\rme^{\cL
t})=s(\cL)$.

It follows immediately from formulas (\ref{2.10}) and from the
inclusion $\Spe(\cL^0)\subset\{\lambda\in\C;\ \Re\, \lambda\leq
0\}$ that $s(\cL^0)=0$. Inequality (\ref{2.11}) (which implies
$\gamma(\rme^{\cL^0t})\leq 0$) and the general inequality
$\gamma(\rme^{\cL^0t})\geq s(\cL^0)$ yield
$\gamma(\rme^{\cL^0t})=0$. Furthermore, $\widetilde{r}_{\rm
ess}(\rme^{\cL^0t})\geq 1$ for each $t>0$, because
$\widetilde{r}_{\rm ess}(\rme^{\cL^0t})$ is the infimum of the
set of $\rho>0$ such that $\{\zeta\in\Spe(\rme^{\cL^0t});\
|\zeta|>\rho\}$ consists of isolated eigenvalues of
$\rme^{\cL^0t}$ with finite algebraic multiplicities, and each
set $\{\zeta\in\Spe(\rme^{\cL^0t});\ |\zeta|>\rho\}$ (for
$0<\rho<1$) contains $\{\rme^{\lambda t};\
\lambda\in\Spe(\cL^0),\ \ln\rho<\Re\, \lambda\}$ (due to the
spectral inclusion theorem), which is not isolated. On the
other hand, since $\widetilde{\gamma}_{\rm
ess}(\rme^{\cL^0t})\leq \gamma(\rme^{\cL^0t})=0$, we obtain
$\widetilde{\gamma}_{\rm ess}(\rme^{\cL^0t})=0$. We have proven
that
\begin{equation}
s(\cL^0)\ =\ \gamma(\rme^{\cL^0 t})\ =\ \widetilde{\gamma}_{\rm
ess}(\rme^{\cL^0t})\ =\ 0. \label{2.14}
\end{equation}

Let us now focus on operator $\cL$ and the semigroup $\rme^{\cL
t}$. If $\omega=0$ then $s(\cL)= \gamma(\rme^{\cL t})$, because
the semigroup $\rme^{\cL t}$ is analytic. The validity of the same
equality in the case $\omega\not=0$ is, however, a subtler
problem. Nevertheless, since $\cL=\cL^0+B^2$, the semigroup
$\rme^{\cL t}$ satisfies the variation of parameters formula
\begin{equation}
\rme^{\cL t}\ =\ \rme^{\cL^0t}+\int_0^t\rme^{\cL\tau}\, B^2\,
\rme^{\cL^0(t-\tau)}\; \rmd\tau \label{2.16}
\end{equation}
for $t\geq 0$, see e.g.~\cite[p.~161]{EnNa}. (The formula has
been in fact derived in \cite{EnNa} under the assumption that
operator $B^2$ is bounded, but using the ``smoothing''
properties of $\rme^{\cL^0t}$ following from (\ref{2.12}) and
(\ref{2.12a}), and applying the same arguments as in
\cite{EnNa}, one also obtains (\ref{2.12}) for the concrete
unbounded operator $B^2$ we deal with.) The integral on the
right hand side of (\ref{2.16}) converges in the topology of
$\calL(\Ls)$, because
\begin{align*} \bigl\| \rme^{\cL\tau}\, B^2\,
\rme^{\cL^0(t-\tau)}\bfphi \bigr\|_2\ &\leq\ M_{\mu}\,
\rme^{\mu\tau}\ \|B^2\, \rme^{\cL^0(t-\tau)}\bfphi\|_2\ \leq\
M_{\mu}\, \rme^{\mu\tau}\, \cc62\,
\bigl|\rme^{\cL^0(t-\tau)}\bfphi\bigr|_{1,2} \\
&\leq\ M_{\mu}\, \rme^{\mu\tau}\,
\frac{\cc62\br\cc65}{\sqrt{t-\tau}}\ \|\bfphi\|_2
\end{align*}
for all $\bfphi\in\Ls$ and $\mu>\gamma(\rme^{\cL t})$. We shall
further need the next two lemmas:

\begin{lemma} \label{L2}
The operator $B^2\, \rme^{\cL^0(t-\tau)}$ is compact in $\Ls$ for
each $t>0$ and $0\leq\tau<t$.
\end{lemma}

\begin{proof}
Let $t>0$ and $0\leq\tau<t$ be fixed. It follows from the
inequalities (\ref{2.11})--(\ref{2.12a}) that
$\rme^{\cL^0(t-\tau)}$ is a bounded operator from $\Ls$ to
$\bfW^{2,2}(\Omega)\cap\Wns$. In order to complete the proof, we
show that $B^2$ is a compact operator from $\bfW^{2,2}(\Omega)\cap
\Wns$ to $\Ls$. Thus, let $\{\bfphi_n\}$ be a bounded sequence in
$\bfW^{2,2}(\Omega)\cap \Wns$ and $R>0$. Then
\begin{align}
\|B^2\bfphi_n-B^2\bfphi_m\|_2\ &=\ \bigl\| \bfU\cdot\nabla
(\bfphi_n-\bfphi_m)+(\bfphi_n-\bfphi_m)\cdot\nabla\bfU\bigr\|_2
\nonumber \\ \noalign{\vskip 2pt}
&\leq\ \bigl\| \bfU\cdot\nabla
(\bfphi_n-\bfphi_m)+(\bfphi_n-\bfphi_m)\cdot\nabla\bfU\bigr\|_{2;\,
\Omega_R} \nonumber \\
& \hspace{15pt} +\bigl\| \bfU\cdot\nabla(\bfphi_n-\bfphi_m)+
(\bfphi_n- \bfphi_m)\cdot\nabla\bfU\bigr\|_{2;\, \Omega^R}
\label{2.25}
\end{align}
for $m,n\in\N$. Let $\epsilon>0$ be given. Due to the assumptions
(\ref{1.3}) on function $\bfU$, there exists $R>0$ so large that
the second term on the right hand side is less than or equal to
$\epsilon/2$, independently of $m$ and $n$. Applying the compact
imbedding $\bfW^{2,2}(\Omega_R)\hookrightarrow\hookrightarrow
\bfW^{1,2}(\Omega_R)$ and the boundedness of the operator
$\bfphi\mapsto[\bfU\cdot\nabla\bfphi+\bfphi\cdot\nabla\bfU]$ from
$\bfW^{1,2}(\Omega_R)$ to $\bfL^2(\Omega_R)$ (which follows from
(\ref{1.4})), one can show that there exists a subsequence of
$\{\bfphi_n\}$ (which we denote again by $\{\bfphi_n\}$), such
that the sequence $\bigl\{ \bfU\cdot\nabla
\bfphi_n+\bfphi_n\cdot\nabla\bfU \bigr\}$ converges in
$\bfL^2(\Omega_R)$. Hence the first term on the right hand side of
(\ref{2.25}) is also less than or equal to $\epsilon/2$ for $m$
and $n$ sufficiently large. This shows that $\{B^2\bfphi_n\}$ is a
Cauchy sequence in $\Ls$. The proof is completed.
\end{proof}

\begin{lemma} \label{L3}
Let $t>0$. Then the operator $\int_0^t \rme^{\cL\tau}\, B^2\,
\rme^{\cL^0(t-\tau)}\; \rmd\tau$ is compact in $\Ls$.
\end{lemma}

\begin{proof}
Let us at first show that for each given $\bfphi\in\Ls$, the
function $\rme^{\cL\tau}\, B^2\, \rme^{\cL^0(t-\tau)}\bfphi$ is
continuous (in the norm of $\Ls$) in dependence on $\tau$ in
the interval $[0,t)$. Thus, let $\tau\in[0,t)$ be fixed and let
$\delta$ satisfy $-\tau<\delta<t-\tau$. We have
\begin{align*}
\bigl\| \rme^{\cL(\tau+\delta)}\, & B^2\,
\rme^{\cL^0(t-\tau-\delta)}\bfphi- \rme^{\cL\tau}\, B^2\,
\rme^{\cL^0(t-\tau)}\bfphi \bigr\|_2 \\ \noalign{\vskip 1pt}
&\leq\ \bigl\| \rme^{\cL(\tau+\delta)}\, B^2\, (\rme^{\cL^0(t-
\tau-\delta)}-\rme^{\cL^0(t-\tau)})\bfphi\bigr\|_2+ \bigl\|
(\rme^{\cL(\tau+\delta)}-\rme^{\cL\tau})\, B^2\,
\rme^{\cL^0(t-\tau)}\bfphi\bigr\|_2.
\end{align*}
The second term on the right hand side tends to zero for
$\delta\to 0$ due to the strong continuity of the semigroup
$\rme^{\cL t}$. The first term on the right hand side is less
than or equal to
\begin{align*}
M_{\mu}\, & \rme^{\mu(\tau+\delta)}\ \bigl\| B^2\, (\rme^{\cL^0(t-
\tau-\delta)}-\rme^{\cL^0(t-\tau)})\bfphi\bigr\|_2\ \leq\
M_{\mu}\, \rme^{\mu(\tau+\delta)}\, \cc62\, \bigl| (\rme^{\cL^0(t-
\tau-\delta)}-\rme^{\cL^0(t-\tau)})\bfphi\bigr|_{1,2} \\
\noalign{\vskip 1pt}
&=\ M_{\mu}\, \rme^{\mu(\tau+\delta)}\, \cc62\, \bigl|
\rme^{\cL^0(t-\delta)/2}\, (\rme^{\cL^0((t-\tau)/2-\delta)}-
\rme^{\cL^0(t-\tau)/2})\bfphi\bigr|_{1,2} \\
&\leq\ M_{\mu}\, \rme^{\mu(\tau+\delta)}\, \cc62\,
\frac{\cc65\sqrt{2}}{\sqrt{t-\tau}}\, \bigl\|
(\rme^{\cL^0((t-\tau)/2-\delta)}-
\rme^{\cL^0(t-\tau)/2})\bfphi\bigr\|_2,
\end{align*}
which tends to zero for $\delta\to 0$ due to the strong
continuity of the semigroup $\rme^{\cL^0t}$. The continuity of
$\rme^{\cL\tau}\, B^2\, \rme^{\cL^0(t-\tau)}\bfphi$ in
dependence on $\tau$ is proven.

Thus, $\rme^{\cL\tau}\, B^2\, \rme^{\cL^0(t-\tau)}$ is a family of
compact linear operators in $\Ls$, strongly continuous in
dependence on $\tau$ for $\tau\in[0,\xi]$ for every $0<\xi<t$.
This information enables us to apply Theorem C.7 from
\cite[p.~525]{EnNa} and conclude that $\int_0^{t-\xi}
\rme^{\cL\tau}\, B^2\, \rme^{\cL^0(t-\tau)}\; \rmd\tau$ is a
compact operator in $\Ls$ for every $\xi\in(0,t)$. Since
\vspace{-8pt}
\begin{displaymath}
\int_0^t \rme^{\cL\tau}\, B^2\, \rme^{\cL^0(t-\tau)}\;
\rmd\tau\ =\ \lim_{\xi\to 0+}\ \int_0^{t-\xi} \rme^{\cL\tau}\,
B^2\, \rme^{\cL^0(t-\tau)}\; \rmd\tau
\end{displaymath}
in the topology of $\calL(\Ls)$, and the subspace of compact
linear operators in $\Ls$ is closed in $\calL(\Ls)$, we observe
that the operator $\int_0^t \rme^{\cL\tau}\, B^2\,
\rme^{\cL^0(t-\tau)}\; \rmd\tau$ is compact, too.
\end{proof}

\vspace{4pt}
Formula (\ref{2.16}) and Lemma \ref{L3} show that, for any
$t>0$, the operators $\rme^{\cL t}$ and $\rme^{\cL^0 t}$ differ
just by an additive compact operator. Thus, $\SpessEN(\rme^{\cL
t})= \SpessEN(\rme^{\cL^0t})$. Consequently,
$\widetilde{r}_{\rm ess}(\rme^{\cL t})=\widetilde{r}_{\rm
ess}(\rme^{\cL^0t})$ for each $t>0$ and
$\widetilde{\gamma}_{\rm ess}(\rme^{\cL
t})=\widetilde{\gamma}_{\rm ess}(\rme^{\cL^0t})=0$. (The last
identity is a part of (\ref{2.14}).) Since $\gamma(\rme^{\cL
t})=\max\{\widetilde{\gamma}_{\rm ess}(\rme^{\cL t});\
s(\cL)\}=\max\{0;\ s(\cL)\}$, we obtain the equality
\begin{equation}
\gamma(\rme^{\cL t})\ =\ s(\cL). \label{2.26}
\end{equation}

\begin{theorem} \label{T1}
The uniform growth bound $\gamma(\rme^{\cL t})$ of the
semigroup $\rme^{\cL t}$ and the spectral bound $s(\cL)$ of
operator $\cL$ are equal. Moreover, for every $\xi>0$, the set
$\Gamma_{\xi}:=\Spe(\cL)\cap\{\lambda\in\C;\ \Re\,
\lambda\geq\xi\}$ consists of at most a finite number of
eigenvalues of $\cL$ with finite algebraic multiplicities.
\end{theorem}

\begin{proof}
The equality of $\gamma(\rme^{\cL t})$ and $s(\cL)$ has already
been proven. Let $\xi>0$. We may suppose without loss of
generality that $s(\cL)>0$ and $\xi<s(\cL)$. Assume, by
contradiction, that the set $\Gamma_{\xi}$ is infinite. The
elements of $\Gamma_{\xi}$ cannot accumulate at any point of
$\C$, because it would contradict the description of
$\Spe(\cL)$ given in subsection \ref{Se2}.4. Their real parts
are in a bounded interval $[\xi,s(\cL)]$, so the real parts
have a cluster point $\xi_0\in[\xi,s(\cL)]$. Thus, the set
$\exp(t\, \Gamma_{\xi})$ (for some $t>0$) has a cluster point
on the circle $|z|=\rme^{\xi_0 t}$ in $\C$. This is, however,
impossible, because the cluster point of $\exp(t\,
\Gamma_{\xi})$ is also a cluster point of $\exp[t\,
\Spe(\cL)]$, i.e.~also a cluster point of $\Spe(\rme^{\cL t})$
(due to the inclusion $\exp[t\, \Spe(\cL)]\subset\Spe(\rme^{\cL
t})$) and since $\rme^{\xi_0 t}>1=\widetilde{r}_{\rm
ess}(\rme^{\cL^0t})=\widetilde{r}_{\rm ess}(\rme^{\cL t})$, the
set $\Spe(\rme^{\cL t})$ cannot have a cluster point on the
circle $|z|=\rme^{\xi_0 t}$.
\end{proof}

\section{Spectral instability of the zero solution of equation
(\ref{2.2})} \label{Se3}

The purpose of this section is to prove the next theorem:

\begin{theorem} \label{T2}
Assume that $\Spe(\cL)\cap\{\lambda\in\C;\ \Re\,
\lambda>0\}\not=\emptyset$. Then the zero solution of equation
(\ref{2.2}) is unstable in the sense that there exists
$\epsilon>0$ such that to any $\delta>0$ there exists $T>0$,
$\Tstar\in(0,T)$ and a solution $\bfvstar$ of equation
(\ref{2.2}) on the time interval $(0,T)$ such that
$\|\bfvstar(0)\|_{1,2}\leq\delta$ and
$\bfvstar(\Tstar)\|_2\geq\epsilon$. Consequently, the steady
solution $\bfU$ of the problem (\ref{1.1a})--(\ref{1.1d})
(satisfying (\ref{1.3}) and (\ref{1.4})) is unstable in the
same sense.
\end{theorem}

\noindent
We present the proof in eight steps, which are explained in
subsections \ref{Se3}.1--\ref{Se3}.8. For reader's convenience,
we focus just on the main ideas in this section, leaving the
detailed explanation and derivation of technical arguments and
estimates to Section \ref{Se4}.

\vspace{4pt} \noindent
{\bf \ref{Se3}.1. Concrete solutions $\bfvr$ and $\bfvi$ of
equation (\ref{2.2}) and function $\bfv$.} \ It follows from
Theorem \ref{T1} that $\Spe(\cL)\cap\{\lambda\in\C;\ \Re\,
\lambda>0\}$ consists only of eigenvalues of $\cL$ and that one
can choose an eigenvalue with the largest real part, equal to
$s(\cL)$. Let $a+\rmi\br b$ be such an eigenvalue. Let
$\bfzeta+\rmi\br\bfeta$ be a corresponding eigenfunction. (The
numbers $a$, $b$ and the functions $\bfzeta$, $\bfeta$ are
supposed to be real.) The eigenfunction can be normalized so
that $\|\bfzeta+\rmi\bfeta\|_{1,2}=1$. Obviously,
\begin{displaymath}
\|\rme^{\cL t}\, (\bfzeta+\rmi\br\bfeta)\|_{1,2}\ =\
\|\rme^{(a+\rmi\br b)t}(\bfzeta+\rmi\br\bfeta)\|_{1,2}\ =\
\rme^{at}\, \|\bfzeta+\rmi\br\bfeta\|_{1,2}\ =\ \rme^{at}.
\end{displaymath}

Let $\delta>0$ and $\bfvr$ and $\bfvi$ be solutions of the
equation (\ref{2.2}), satisfying the initial conditions
$\bfvr(0)=\delta\br\bfzeta$ and $\bfvi(0)=\delta\br\bfeta$,
respectively. It follows from the existential results, cited in
subsection \ref{Se2}.3, that both the solutions $\bfvr$ and
$\bfvi$ exist on the time interval $[0,T)$ (for some $T>0$),
belong to the class (\ref{2.3}) and either $T=\infty$ or
$\|\bfvr(t)\|_{1,2}+\|\bfvi(t)\|_{1,2}\to\infty$ for $t\to T-$.
Put $\bfv:=\bfvr+\rmi\br\bfvi$. Function $\bfv$ satisfies the
equation
\begin{equation}
\frac{\rmd\bfv}{\rmd t}\ =\ \cL\bfv+\cN\bfvr+\rmi\, \cN\bfvi
\label{3.4}
\end{equation}
and the initial condition
\begin{equation}
\bfv(0)\ =\ \delta\br(\bfzeta+\rmi\br\bfeta). \label{3.5}
\end{equation}

\noindent
{\bf \ref{Se3}.2. Numbers $K$ and $\Tone$.} \  Let $K>1$. Since
$\|\bfv(0)\|_{1,2}=\delta$, the inequality
\begin{equation}
\|\bfv(t)\|_{1,2}\ \leq\ \delta K\, \rme^{at} \label{3.9}
\end{equation}
holds for $t$ in some right neighborhood of $0$. Denote by
$\Tone$ the maximum number such that (\ref{3.9}) holds for all
$t\in(0,\Tone)$. ($\Tone=\infty$ is also admitted.)

From now on, all estimates, cited or derived in subsections
\ref{Se3}.5 and \ref{Se4}.3--\ref{Se4}.5, are related to $t$ in
the time interval $[0,\Tone]$ (if $\Tone<\infty$) or
$[0,\infty)$ (if $\Tone=\infty$).

\vspace{4pt} \noindent
{\bf \ref{Se3}.3. Decomposition of the space $\Ls$.} \ Let
$\kappa>0$ be fixed. It is proven in \cite{Ne3} that the number
of positive eigenvalues of the self-adjoint operator $\nu
A+(1+\kappa)B^2_s$ is finite. Let us denote these eigenvalues
by $\lambda_1\leq\lambda_2\leq \dots\leq\lambda_N$, each of
them being repeated as many times as is its multiplicity. Let
$\bfphi_1,\dots,\bfphi_N$ be associated eigenfunctions. We can
assume that the eigenfunctions have been chosen so that they
constitute an orthonormal system in $\Ls$. Denote by $\Ls'$ the
linear hull of $\bfphi_1,\dots,\bfphi_N$ and by $P'$ the
orthogonal projection of $\Ls$ onto $\Ls'$. The orthogonal
complement to $\Ls'$ in $\Ls$ is denoted by $\Ls''$ and the
orthogonal projection of $\Ls$ onto $\Ls''$ is denoted by
$P''$. Then we have
\begin{displaymath}
\Ls\ =\ \Ls'\oplus\Ls''
\end{displaymath}
and the operator $\, \nu A+(1+\kappa)B^2_s\, $ is reduced on each
of the subspaces $\Ls'$ and $\Ls''$. Using the negative
definiteness of $\nu A+(1+\kappa)B^2_s$ in $\Ls''$, one can easily
derive that $\cn71$
\begin{equation}
\bigl( (\nu A+B^2_s)\bfphi,\bfphi\bigr)_2\ \leq\ -\cc71\,
|\bfphi|_{1,2}^2 \label{3.1}
\end{equation}
for all $\bfphi\in\Ls''\cap D(A)$, where
$\cc71=\kappa\nu/(1+\kappa)$. Inequality (\ref{3.1}) shows that
the operator $\nu A+B^2_s$ is essentially dissipative in $\Ls''$.

\vspace{4pt} \noindent
{\bf \ref{Se3}.4. Splitting of the problem (\ref{3.4}),
(\ref{3.5}).} \ We show in subsection \ref{Se4}.2 that the
solution $\bfv$ of (\ref{3.4}), (\ref{3.5}) can be expressed in
the form $\bfv=\bfw+\bfz$, where $\bfw,\, \bfz$ are solutions
of the equations
\begin{alignat}{3}
& \frac{\rmd\bfw}{\rmd t}-\omega\br B^0\bfw-u_{\infty}\br B^1\bfw\
&& =\ \nu A\bfw+B^2_s\bfw-P'[\omega\br B^0\bfw+
u_{\infty}\br B^1\bfw-\kappa B^2_s\bfw] \nonumber \\
& && \hspace{15pt} +P''B^2_a\bfw+ P''[\cN\bfvr+ \rmi\br\cN\bfvi],
\label{3.6} \\ \noalign{\vskip 1pt}
& \frac{\rmd\bfz}{\rmd\br t}-\omega\br B^0\bfz-u_{\infty}\br
B^1\bfz\ && =\ \nu A\bfz+B^2\bfz+P'\bigl[ \omega\br
B^0\bfw+u_{\infty}\br B^1\bfw-\kappa\br B^2_s\bfw+B^2_a\bfw \bigr]
\nonumber \\
& && \hspace{15pt} +P'[\cN\bfvr+\rmi\br\cN\bfvi] \label{3.7}
\end{alignat}
with the initial conditions
\begin{equation}
\bfw(0)=\bfzero, \qquad \bfz(0)=\delta\br(\bfzeta+\rmi\br\bfeta)
\label{3.8}
\end{equation}
on the interval $(0,T)$ in the class (\ref{2.3}). We also show in subsection \ref{Se4}.2 that
(\ref{3.6}) is an equation in the space $\Ls''$, where the
operator $\nu A+B^2_s$ is essentially dissipative. On the other
hand, all terms on the right hand side of equation (\ref{3.7}),
except for $\nu A\bfz+B^2\bfz$, belong to the
finite--dimensional space $\Ls'$, where all norms are
equivalent. These properties play an important role in the
estimates of $\bfw$ and $\bfz$ (and therefore also estimates of
$\bfv$), derived in subsections \ref{Se4}.3--\ref{Se4}.5.

Note that the belonging of $\bfw$ and $\bfz$ to the class
(\ref{2.3}) is not sufficient to guarantee that the term
$P'B^0\bfw$ in equation (\ref{3.7}) has a sense, because
$\bfw(t)$ need not generally be in $D(\cL^0)$ for
a.a.~$t\in(0,T)$. Nevertheless, we also explain in subsection
\ref{Se4}.1 how one should understand the meaning of
$P'B^0\bfw$.

\vspace{4pt} \noindent
{\bf \ref{Se3}.5. Estimates of functions $\bfz$, $\bfw$ and
$\bfv$.} \ Equation (\ref{3.7}) can also be written in the form
\begin{displaymath}
\frac{\rmd\bfz}{\rmd\br t}\ =\ \cL\bfz+P'\bigl[ \omega\br
B^0\bfw+u_{\infty}\br B^1\bfw-\kappa\br B^2_s\bfw+B^2_a\bfw
\bigr]+P'[\cN\bfvr+\rmi\br\cN\bfvi].
\end{displaymath}
Using the integral representation of $\bfz$ by means of the
variation of parameters formula, we get $\bfz=\bfz_1+\bfz_2$,
where\begin{equation}
\begin{array}{ll}\medskip\label{NG}
\bfz_1(t)\ &=\ \rme^{\cL t}\br\bfz(0)\ =\ \delta\,
\rme^{(a+\rmi b)t}\, (\bfzeta+\rmi\br\bfeta), \\
\bfz_2(t)\ &\displaystyle{=\ \int_0^t \rme^{\cL(t-\tau)}\, P' [\omega\br
B^0\bfw(\tau)+u_{\infty}\br B^1\bfw(\tau)-\kappa\br B^2_s
\bfw(\tau)+B^2_a\bfw(\tau)]\; \rmd\tau} \\
& \displaystyle{\hspace{15pt} +\int_0^t\rme^{\cL(t-\tau)}\, P'[\cN\bfvr(\tau)+
\cN\bfvi(\tau)]\; \rmd\tau.}
\end{array}
\end{equation}
The function $\bfz_1$ satisfies the obvious equality
\begin{equation}
\|\bfz_1(t)\|_2\ =\ \delta\, \|\bfzeta+\rmi\br\bfeta\|_2\
\rme^{at}. \label{3.29}
\end{equation}
Let $\mu\in(a,2a)$. Since $a=s(\cL)=\gamma(\rme^{\cL t})$, there
exists $M_{\mu}>0$ such that
\begin{equation}
\|\rme^{\cL t}\bfphi\|_2\ \leq\ M_{\mu}\, \rme^{\mu t}\,
\|\bfphi\|_2 \label{3.3}
\end{equation}
for all $\bfphi\in\Ls$ and $t\geq 0$. Applying in \eqref{NG}$_2$ this inequality,
along with (\ref{3.3}) and the estimates from subsection
\ref{Se4}.1, we obtain $\cn72\cn73$
\begin{equation}
\|\bfz_2(t)\|_2\ \leq\ \int_0^t M_{\mu}\, \rme^{\mu(t-\tau)}\,
\cc72\, \|\bfw(\tau)\|_2\; \rmd\tau+\frac{M_{\mu}\, \cc73\,
\delta^2 K^2}{2a-\mu}\ \rme^{2at}, \label{3.30}
\end{equation}
where $\cc72$ and $\cc73$ are appropriate positive constants.
The function $\bfw$ can be estimated from equation (\ref{3.6}),
multiplying (\ref{3.6}{) by $\bfw$ and integrating in $\Omega$,
see subsection \ref{Se4}.3. We get: $\cn76$
\begin{equation}
\|\bfw(t)\|_2^2+\cc71\int_0^t|\bfw(\tau)|_{1,2}^2\; \rmd\tau\
\leq\ \frac{\cc76\, \delta^4\br K^4}{4a}\ \rme^{4at}.
\label{3.16}
\end{equation}
From (\ref{3.30}) and 
(\ref{3.16}), we get, in particular, $\cn77\cn78$
\begin{equation}
\|\bfz_2(t)\|_2\ \leq\ \cc77\, (\delta K\, \rme^{at})^2,
\label{3.13}
\end{equation}
which, combined with $\eqref{NG}_1$,   yields
\vspace{-10pt}
\begin{equation}
\|\bfz(t)\|_2\ \leq\ \delta\, \|\bfzeta+\rmi\br\bfeta\|_2\
\rme^{at}+\cc77\, (\delta K\, \rme^{at})^2. \label{3.17}
\end{equation}
Similarly, we also derive the inequality
\begin{equation}
\|\bfz(t)\|_{1,2}\ \leq\ \delta\, \|\bfzeta+\rmi\br\bfeta\|_{1,2}\
\rme^{at}+\cc78\, (\delta K\, \rme^{at})^2\ \equiv\ \delta\,
\rme^{at}+\cc78\, (\delta K\, \rme^{at})^2; \label{3.18}
\end{equation}
see Subsection \ref{Se4}.4 for the details.

We also obtain an estimate of $|\bfw|_{1,2}$, by multiplying
equation (\ref{3.6}) by $(-A\bfw)$, integrating in $\Omega$ and
using inequalities (\ref{2.5a}), (\ref{2.5b}), (\ref{2.5f}) and
(\ref{3.9}). Summing then appropriately the estimates of
$\|\bfz\|_2$, $|\bfz|_{1,2}$, $\|\bfw\|_2$ and $|\bfw|_{1,2}$,
we get $\cn81\cn82\cn83\cn84$
\begin{equation}
\|\bfv(t)\|_{1,2}\ \leq\ \cc81\, (\delta\, \rme^{at})+
(\cc82+\cc83\, K)\, K^2\, (\delta\, \rme^{at})^2+\cc84\, K^3\,
(\delta\, \rme^{at})^3 \label{3.26},
\end{equation}
where the constants $\cc81$--$\cc84$ are independent of
$\delta$ and $K$. (See subsection \ref{Se4}.5.) Note that
$\cc81>1$.

\vspace{4pt} \noindent
{\bf \ref{Se3}.6. Choice of the Number $\Tstar$.} \ The right hand side of
(\ref{3.26}) is less than the right hand side of (\ref{3.9}) at
the initial time $t=0$ if
\vspace{-4pt}
\begin{displaymath}
\cc81+(\cc82+\cc83\, K)\, K^2\, \delta+\cc84\, K^3\, \delta^2\
<\ K,
\end{displaymath}
which is satisfied if  $K$ and $\delta$ are chosen
so that
\begin{equation}
\cc81<K, \qquad (\cc82+\cc83\, K)\, K^2\, \delta+\cc84\, K^3\,
\delta^2<K-\cc81. \label{3.35}
\end{equation}
Assuming that (\ref{3.35}) holds and using the fact that the
right hand side of (\ref{3.26}) growths with increasing $t$
faster than the right hand side of (\ref{3.9}), we deduce that
there exists $0<\Tstar\leq \Tone$ such that the right hand
sides of (\ref{3.26}) and (\ref{3.9}) coincide at the time
$\Tstar$. It means that
\begin{displaymath}
\cc81\, (\delta\, \rme^{a\Tstar})+ (\cc82+\cc83\, K)\, K^2\,
(\delta\, \rme^{a\Tstar})^2+\cc84\, K^3\, (\delta\,
\rme^{a\Tstar})^3\ =\ \delta\br K\, \rme^{a\Tstar}.
\end{displaymath}
This yields $\cn85$
\vspace{-8pt}
\begin{equation}
\delta\, \rme^{a\Tstar}\ =\ \cc85(K)\ :=\ \frac{2\, (K-\cc81)}{
\cc84\br K^2\, [(\cc82+\cc83\br K)+ \sqrt{(\cc82+\cc83\br
K)^2+4(K-\cc81)}]}. \label{3.32}
\end{equation}

\vspace{4pt} \noindent
{\bf \ref{Se3}.7. Lower estimates of $\|\bfv(\Tstar)\|_2$.} \
(\ref{3.29}) and (\ref{3.13} yield
\begin{displaymath}
\|\bfz(\Tstar)\|_2\ \geq\
\|\bfz_1(\Tstar)\|_2-\|\bfz_2(\Tstar)\|_2\ \geq\ \delta\,
\|\bfzeta+\rmi\br\bfeta\|_2\, \rme^{a\Tstar}-\cc77\,
\delta^2\br K^2\, \rme^{2a\Tstar}.
\end{displaymath}
Hence, due to (\ref{3.16}), we also have $\cn89$
\begin{align*}
\|\bfv(\Tstar)\|_2\ &\geq\ \|\bfz(\Tstar)\|_2-\|\bfw(\Tstar)\|_2\ \geq\
\delta\, \|\bfzeta+\rmi\br\bfeta\|_2\, \rme^{a\Tstar}-\cc89\,
\delta^2\br K^2\, \rme^{2a\Tstar},
\end{align*}
where $\cc89=\cc77+\sqrt{\cc76}/(2\sqrt{a})$. Expressing
$\delta\, \rme^{a\Tstar}$ from (\ref{3.32}), we obtain $\cn91$
\begin{equation}
\|\bfv(\Tstar)\|_2\ \geq\ \cc85(K)\, \bigl[ \br
\|\bfzeta+\rmi\br\bfeta\|_2-\cc89\br K^2\, \cc85(K) \bigr]\ =:\
\cc91(K). \label{3.33}
\end{equation}
If $K>\cc81$ is chosen sufficiently close to $\cc81$ then
$\|\bfzeta+\rmi\br\bfeta\|_2>\cc89\br K^2\, \cc85(K)$, which
means that $\cc91(K)$ is positive. It is remarkable that it is
independent of $\delta$.

\vspace{4pt} \noindent
{\bf \ref{Se3}.8. Completion of the proof.} \ Recall that
$\bfv=\bfvr+\rmi\br\bfvi$, where $\bfvr$ and $\bfvi$ are real
solutions of equation (\ref{2.2}), satisfying the initial
conditions $\bfvr(0)=\delta\br\bfzeta$ and
$\bfvi(0)=\delta\br\bfeta$, respectively. Put
$\epsilon:=\cc91(K)/\sqrt{2}$. Inequality (\ref{3.33}) implies
that either $\|\bfvr(\Tstar)\|_2\geq\epsilon$ or
$\|\bfvi(\Tstar)\|_2\geq\epsilon$. Thus, given $\delta>0$
arbitrarily small (satisfying (\ref{3.35})), there exists a
real solution $\bfvstar$ of equation (\ref{2.2})
(i.e.~$\bfvstar=\bfvr$ or $\bfvstar=\bfvi$) whose initial
$\bfW^{1,2}$--norm is less than or equal to $\delta$ and the
$\bfL^2$--norm at the time $t=\Tstar$ is greater than or equal
to $\epsilon$. This completes the proof of Theorem \ref{T2}.
\hfill \K

\section{Appendix} \label{Se4}

{\bf \ref{Se4}.1. Estimate (\ref{3.30}).} \ The crucial point in
the derivation of (\ref{3.30}) are the inequalities
\begin{equation}
\int_{\Omega}|\bfx|^2\, |\Delta\bfphi_k|^2\;
\rmd\bfx+\int_{\Omega}|\bfx|^2\, |\nabla\bfphi_k|^2\; \rmd\bfx\ <\
\infty \qquad (k=1,\dots,N), \label{4.1}
\end{equation}
see \cite[Lemma 7]{GaNe1}. They enable one to show that $\cn42$
\begin{equation}
\|P'B^0\bfphi\|_2+\|P'B^1\bfphi\|_2+\|P'B^2_s\bfphi\|_2+
\|P'B^2_a\bfphi\|_2\ \leq\ \cc42\, |\bfphi|_{1,2} \label{4.2}
\end{equation}
for all $\bfphi\in D(A)$, see \cite{GaNe1}. Using especially
the fact that $P'$ is the projection onto the $N$--dimensional
space $\Ls'$, where the norms $\|\, .\, \|_2$ and $|\, .\,
|_{1,2}$ are equivalent, and mainly copying the procedure from
\cite{GaNe1}, one can show that (\ref{4.2}) is also satisfied
with $\cn43\cc43\, \|\bfphi\|_2$ on the right hand side instead
of $\cc42\, |\bfphi|_{1,2}$. Inequalities (\ref{4.1}) imply that the term $P'B^0\bfphi$ is well defined, although $B^0\bfphi$ is
not necessarily in $\Ls$ for $\bfphi\in D(A)$. (Recall that the
inclusion $B^0\bfphi\in\Ls$ is guaranteed if $\bfphi\in
D(\cL^0)$.) Nevertheless, even for $\bfphi\in D(A)$, one can
put
\vspace{-8pt}
\begin{equation}
P'B^0\bfphi\ :=\ \sum_{k=1}^N\biggl(\int_{\Omega}
[(\bfe_1\times\bfx)\cdot\nabla\bfphi-\bfe_1\times\bfphi]\cdot
\bfphi_k\; \rmd\bfx\biggr)\, \bfphi_k, \label{4.3}
\end{equation}
where the integral over $\Omega$ equals
\begin{align*}
\lim_{R\to\infty}\ & \int_{\Omega_R}[(\bfe_1\times\bfx)\cdot
\nabla\bfphi-\bfe_1\times\bfphi]\cdot\bfphi_k\; \rmd\bfx \\
&=\ \lim_{R\to\infty}\ \biggl( \int_{\partial\Omega_R}[(\bfe_1
\times\bfx)\cdot\bfn]\, \bfphi\cdot\bfphi_k\; \rmd S-\int_{\Omega_R}
[(\bfe_1\times\bfx)\cdot\nabla\bfphi_k-\bfe_1\times\bfphi_k]\cdot
\bfphi\; \rmd\bfx \biggr) \\
&=\ \int_{\Omega} [(\bfe_1\times
\bfx)\cdot\nabla\bfphi_k-\bfe_1\times\bfphi_k]\cdot\bfphi\; \rmd\bfx.
\end{align*}
(The surface integral over $\partial\Omega_R$ equals zero
because the integrand is equal to zero a.e.~in
$\partial\Omega_R$.) Inequalities (\ref{4.1}) guarantee the
convergence of the last integral.

The term involving the nonlinear operator $\cN$ can be estimated as
follows: $\cn44$
\begin{align}
\| P'\cN\bfphi\|_2\ &=\ \sup_{\boldsymbol{\psi}\in\bfL^2(\Omega)}\
\frac{\bigl| (P'\cN\bfphi,\bfpsi)_2\bigr|}{\|\bfpsi\|_2}\ =\
\sup_{\boldsymbol{\psi}\in\Ls'}\ \frac{\bigl|
(\cN\bfphi,\bfpsi)_2\bigr|}{\|\bfpsi\|_2} \nonumber \\
&=\ \sup_{\boldsymbol{\psi}\in \Ls'}\ \frac{1}{\|\bfpsi\|_2}\,
\biggl| \int_{\Omega} \bfphi\cdot\nabla\bfpsi\cdot\bfphi\;
\rmd\bfx \biggr|\ \leq\ \sup_{\boldsymbol{\psi}\in \Ls'}\
\frac{\|\bfphi\|_4^2\, |\bfpsi|_{1,2}}{\|\bfpsi\|_2} \nonumber \\
\noalign{\vskip 2pt}
&\leq\ \cc44\, \|\bfphi\|_2^{1/2}\, |\bfphi|_{1,2}^{3/2}.
\label{4.4}
\end{align}
(We use H\"older's and Sobolev's inequalities and the inclusion
$\bfpsi\in\Ls'$.) Applying these inequalities to the integrals
in the formula for $\bfz_2(t)$ and estimating the norms
$\|\bfvr\|_{1,2}$ and $\|\bfvi\|_{1,2}$ by means of
(\ref{3.9}), we obtain (\ref{3.30}).

\vspace{4pt} \noindent
{\bf \ref{Se4}.2. The system (\ref{3.6}), (\ref{3.7}).} Let us
 first show that (\ref{3.6}) is an equation in $\Ls''$.
Denote $\bfw':=P'\bfw$ and $\bfw'':=P''\bfw$. We claim that
$\bfw'\equiv\bfzero$. Since $\rmd\bfw'/\rmd t\equiv
P'(\rmd\bfw/\rmd t)$ for a.a.~$t\in(0,T)$ and,  from (\ref{2.3}) and (\ref{4.3}), we have
\begin{displaymath}
P'\br\frac{\rmd\bfw}{\rmd t}\ =\ P'\br\Bigl( \frac{\rmd\bfw}{\rmd
t}-\omega\br B^0\bfw-u_{\infty}\br B^1\bfw\Bigr)+P'\bigl(\omega\br
B^0\bfw+u_{\infty}\br B^1\bfw\bigr)\,,
\end{displaymath}
equation (\ref{3.6}) can
be rewritten as follows
\begin{align*}
\frac{\rmd\bfw'}{\rmd t}+\Bigl(\frac{\rmd\bfw}{\rmd t}-\omega
B^0\bfw-u_{\infty} B^1\bfw\Bigr) &=\, \nu
A\bfw+B^2_s\bfw+P'\br\Bigl(\frac{\rmd\bfw}{\rmd t}-\omega\br
B^0\bfw-u_{\infty}\br B^1\bfw\Bigr) \\
& \hspace{15pt} +\kappa P'B^2_s\bfw+P''B^2_a\bfw+ P''[\cN\bfvr+
\rmi\br\cN\bfvi], \\ \noalign{\vskip 1pt}
\frac{\rmd\bfw'}{\rmd t}+P''\br\Bigl(\frac{\rmd\bfw}{\rmd
t}-\omega B^0\bfw-u_{\infty} B^1\bfw\Bigr) &=\, \bigl[\nu
A\bfw'+(1+\kappa)\br B^2_s\bfw'\bigr]+\bigl[\nu
A\bfw''+(1+\kappa)\br B^2_s\bfw''\bigr] \\
& \hspace{15pt} -\kappa P''B^2_s\bfw+P''B^2_a\bfw+ P''[\cN\bfvr+
\rmi\br\cN\bfvi].
\end{align*}
Projecting the last equation on $\Ls'$ and using the fact that
the operator $\nu A+(1+\kappa)\br B^2_s$ is reduced on $\Ls'$
and $\Ls''$, we obtain
\begin{displaymath}
\frac{\rmd\bfw'}{\rmd t}\ =\ \nu A\bfw'+(1+\kappa)\br B^2_s\bfw'.
\end{displaymath}
This, together with the initial condition $\bfw'(0)=\bfzero$,
yields $\bfw'\equiv\bfzero$.

Assume that $(\bfw,\bfz)$ is a solution of
(\ref{3.6})--(\ref{3.7}). Summing the equations (\ref{3.6}),
(\ref{3.7}) we observe that $\bfv\equiv\bfw+\bfz$ satisfies
equation (\ref{3.4}). Similarly, the sum of the initial
conditions in (\ref{3.8}) yields (\ref{3.5}).

On the other hand, if $\bfv$ is a solution of (\ref{3.4}),
(\ref{3.5}) on the time interval $(0,T)$ then, applying the
same method as in \cite{CuTu}, one can at first solve equation
(\ref{3.6}) with the initial condition $\bfw(0)=\bfzero$ as a
linear problem for the unknown $\bfw$, and afterwards equation
(\ref{3.7}) with the initial condition
$\bfz(0)=\delta\br(\bfzeta+\rmi\br\bfeta)$ as a linear problem
for the unknown $\bfz$. Both problems are uniquely solvable on
the same interval $(0,T)$.

\vspace{4pt} \noindent
{\bf \ref{Se4}.3. Estimate (\ref{3.16}).} \ We multiply
equation (\ref{3.6}) by $\bfw$, integrate in $\Omega$ and apply
the next identity, which comes from \cite[Lemma 1]{GaNe1} and
holds for a.a.~$t\in(0,T)$:
\begin{equation}
\int_{\Omega}\Bigl(\frac{\rmd\bfw}{\rmd\br t}-\omega\br
B^0\bfw-u_{\infty}\br B^1\bfw\Bigr)\cdot\bfw\; \rmd\bfx\ =\
\frac{\rmd}{\rmd\br t}\, \frac{1}{2}\, \|\bfw\|_2^2
\label{2.1a}
\end{equation}
(Lemma 1 in \cite{GaNe1} is in fact formulated for a solution
$\bfv$ of a concrete equation, but the equation is not used in the
proof.) Thus, applying (\ref{2.1a}), (\ref{3.1}) and the
identities $(P''B^2_a\bfw,\bfw)_2=0$ and $\bigl( P'[\omega\br
B^0\bfw+ u_{\infty}\br B^1\bfw+\kappa
B^2_s\bfw],\br\bfw\bigr)_2=0$ (following from the inclusion
$\bfw\in\Ls''$ and the fact that $B^2_a$ is skew symmetric), we
obtain
\begin{equation}
\frac{\rmd}{\rmd t}\, \frac{1}{2}\, \|\bfw\|_2^2\ \leq\ -\cc71\,
|\bfw|_{1,2}^2+(\cN\bfvr,\bfw\bigr)_2+(\cN\bfvi,\bfw\bigr)_2.
\label{3.15}
\end{equation}
Using H\"older's and Sobolev's inequalities, we get
\begin{align*}
(\cN\bfvr,\bfw)_2\ &=\ \int_{\Omega}
\bfvr\cdot\nabla\bfvr\cdot\bfw\; \rmd\bfx\ =\
\int_{\Omega} \bfvr\cdot\nabla\bfw\cdot\bfvr\; \rmd\bfx \\
\noalign{\vskip 1pt}
&\leq\ \|\bfvr\|_4^2\, |\bfw|_{1,2}\ \leq\ \iota\,
|\bfw|_{1,2}^2+c(\iota)\, \|\bfvr\|_4^4\ \leq\ \iota\,
|\bfw|_{1,2}^2+c(\iota)\, \|\bfvr\|_2\, |\bfvr|_{1,2}^3.
\end{align*}
The term $(\cN\bfvi,\bfw)_2$ can be estimated in the same way.
Applying these estimates of $(\cN\bfvr,\bfw)_2$ and
$(\cN\bfvi,\bfw)_2$ to (\ref{3.15}), choosing $\iota$ sufficiently
small and also applying inequality (\ref{3.9}), we obtain
\begin{equation}
\frac{\rmd}{\rmd t}\, \|\bfw\|_2^2+\cc71\, |\bfw|_{1,2}^2\ \leq\
\cc76\, (\delta\br K\, \rme^{at})^4. \label{4.6}
\end{equation}
Integrating (\ref{4.6}) with respect to $t$ and using the initial
condition $\bfw(0)=\bfzero$, we derive (\ref{3.16}).

\vspace{4pt} \noindent
{\bf \ref{Se4}.4. Estimate (\ref{3.18}).} \ In order to
estimate the norm $\|\bfz(t)\|_{1,2}$, we use the inequality
$\cn79$
\begin{equation}
\|\bfphi\|_{2,2}+\|B^0\bfphi\|_2\ \leq\ \cc79\, \bigl(
\|\cL\bfphi\|_2+\|\bfphi\|_2\bigr) \label{4.7}
\end{equation}
for $\bfphi\in D(\cL)$, which follows from \cite{Hi3} or
\cite{FaNe2}. (The inequality is proven for $\cL^0$ instead of
$\cL$ on the right hand side in \cite{Hi3} and \cite{FaNe2},
but it can be easily modified to the form (\ref{4.7}).) As the
integrands in the formula for $\bfz_2$ lie in $D(\cL)$ for
a.a.~$\tau\in(0,t)$, we obtain
\begin{align*}
\|\bfz(t)\|_{1,2}\ &\leq\ \delta\, \|\bfzeta+\rmi\br
\bfeta\|_{1,2}\ \rme^{at} \nonumber \\
& \hspace{15pt} +\cc79\int_0^t \bigl\| \cL\, \rme^{\cL(t-\tau)}\,
P' [\omega\br B^0\bfw(\tau)+u_{\infty}\br B^1\bfw(\tau)-\kappa\br
B^2_s \bfw(\tau)+B^2_a\bfw(\tau)] \bigr\|_2\; \rmd\tau \nonumber
\\
& \hspace{15pt} +\cc79\int_0^t \bigl\| \cL\, \rme^{\cL(t-\tau)}\,
P'[\cN\bfvr(\tau)+ \cN\bfvi(\tau)] \bigr\|_2\; \rmd\tau \nonumber
\\
& \hspace{15pt} +\cc79\int_0^t \bigl\| \rme^{\cL(t-\tau)}\, P'
[\omega\br B^0\bfw(\tau)+u_{\infty}\br B^1\bfw(\tau)-\kappa\br
B^2_s \bfw(\tau)+B^2_a\bfw(\tau)] \bigr\|_2\; \rmd\tau \nonumber
\\
& \hspace{15pt} +\cc79\int_0^t \bigl\| \rme^{\cL(t-\tau)}\,
P'[\cN\bfvr(\tau)+ \cN\bfvi(\tau)] \bigr\|_2\; \rmd\tau.
\end{align*}
Since $\cL$ commutes with $\rme^{\cL(t-\tau)}$ and $P'$ is a
projection onto a finite--dimensional space, we derive
(\ref{3.18}) in the same way as (\ref{3.17}).

\vspace{4pt} \noindent
{\bf \ref{Se4}.5. Estimate (\ref{3.26}).} \ In order to derive an
estimate of $|\bfw|_{1,2}$, we multiply equation (\ref{3.6}) by
$(-A\bfw)$ and use the formula
\begin{equation}
\int_{\Omega}\Bigl(\frac{\rmd\bfw}{\rmd\br t}-\omega\br
B^0\bfw-u_{\infty}\br B^1\bfw\Bigr)\cdot(-A\bfw)\; \rmd\bfx\ =\
\frac{\rmd}{\rmd\br t}\, \frac{1}{2}\, |\bfw|_{1,2}^2,
\label{2.1b}
\end{equation}
which follows from \cite[Lemma1]{GaNe1}, similarly as
(\ref{2.1a}). If we also apply the inequalities (\ref{2.5b}),
(\ref{2.5f}) and (\ref{4.2}), we get $\cn49\cn50$
\begin{align}
\frac{\rmd}{\rmd t}\, \frac{1}{2}\, & |\bfw|_{1,2}^2+\nu\,
\|A\bfw\|_2^2\ =\ -(B^2_s\bfw,A\bfw)_2+ \bigl(P'[\omega
B^0\bfw+u_{\infty} B^1\bfw-\kappa\br B^2_s\bfw],A\bfw\bigr)_2
\nonumber \\ \noalign{\vskip-1pt}
& \hspace{15pt} - (P''B_a\bfw,A\bfw)_2-\bigl(P''[\cN\bfvr+
\cN\bfvi],A\bfw\bigr)_2 \nonumber \\ \noalign{\vskip 2pt}
&\leq\ c\, |\bfw|_{1,2}\, \|A\bfw\|_2+c\, \|\cN\bfvr+
\cN\bfvi\|_2\, \|A\bfw\|_2 \nonumber \\ \noalign{\vskip 2pt}
&\leq\ \frac{\nu}{4}\, \|A\bfw\|_2^2+c(\nu)\, |\bfw|_{1,2}^2+
c\, \|A\bfv\|_2^{1/2}\, |\bfv|_{1,2}^{3/2}\, \|A\bfw\|_2 \nonumber \\
\noalign{\vskip 2pt}
&\leq\ \frac{\nu}{4}\, \|A\bfw\|_2^2+c(\nu)\,
|\bfw|_{1,2}^2+ c\, \|A\bfz\|_2^{1/2}\, |\bfv|_{1,2}^{3/2}\,
\|A\bfw\|_2+ c\, |\bfv|_{1,2}^{3/2}\, \|A\bfw\|_2^{3/2} \nonumber \\
\noalign{\vskip 2pt}
&\leq\ \frac{\nu}{2}\, \|A\bfw\|_2^2+c(\nu)\,
|\bfw|_{1,2}^2+ \frac{\nu}{2}\, \|A\bfz\|_2^2+c(\iota,\nu)\, \,
|\bfv|_{1,2}^6 \nonumber \\ \noalign{\vskip 2pt}
&\leq\ \frac{\nu}{2}\, \|A\bfw\|_2^2+c(\nu)\,
|\bfw|_{1,2}^2+\frac{\nu}{2}\, \|A\bfz\|_2^2+c(\nu)\,
(\delta K\, \rme^{at})^6, \nonumber \\ \noalign{\vskip 2pt}
\frac{\rmd}{\rmd t}\, & |\bfw|_{1,2}^2+\nu\, \|A\bfw\|_2^2\ \leq\
\cc49\, |\bfw|_{1,2}^2+\nu\, \|A\bfz\|_2^2+\cc50\, \, (\delta
K\, \rme^{at})^6, \label{4.5}
\end{align}
where $\cc49=\cc49(\nu)$ and $\cc50=\cc50(\nu)$.

In order to get rid of the term $ \|A\bfz\|_2^2$ on the
right hand side of (\ref{4.5}), we multiply equation (\ref{3.7})
by $(-A\bfz)$ and integrate in $\Omega$. Applying formula
(\ref{2.1b}) and inequalities (\ref{3.9}), (\ref{3.18}),
(\ref{4.2}) and (\ref{4.3}), we obtain $\cn33\cn34\cn35$
\begin{align}
\frac{\rmd}{\rmd t}\, \frac{1}{2}\, |\bfz|_{1,2}^2+\nu\,
\|A\bfz\|_2^2\ &=\ -(B^2\bfz,A\bfz)_2-\bigl(P'[\omega\br
B^0\bfw+u_{\infty}\br B^1\bfw-\kappa\br B^2_s\bfw+
B^2_a\bfw],A\bfz\bigr)_2 \nonumber \\
& \hspace{15pt} -\bigl(P'[\cN\bfvr+\cN\bfvi],A\bfz\bigr)_2
\nonumber \\ \noalign{\vskip 2pt}
&\leq\ \cc62\, |\bfz|_{1,2}\, \|A\bfz\|_2+c\, |\bfw|_{1,2}\,
\|A\bfz\|_2+c\, \|\bfv\|_2^{1/2}\, |\bfv|_{1,2}^{3/2}\,
\|A\bfz\|_2 \nonumber \\ \noalign{\vskip 2pt}
&\leq\ \frac{\nu}{2}\, \|A\bfz\|_2^2+c(\nu)\,
|\bfz|_{1,2}^2+c(\nu)\, |\bfw|_{1,2}^2+c(\nu)\, \|\bfv\|_2\,
|\bfv|_{1,2}^3, \nonumber \\
\frac{\rmd}{\rmd t}\, |\bfz|_{1,2}^2+\nu\, \|A\bfz\|_2^2\ &\leq\
\cc33\, \bigl[ \delta^2\, \rme^{2at}+2\cc78\, \delta\, \rme^{at}\,
(\delta\br K\, \rme^{at})^2+\cc78^2\, (\delta\br K\, \rme^{at})^4
\bigr] \nonumber \\ \noalign{\vskip-2pt}
& \hspace{15pt} +\cc34\, |\bfw|_{1,2}^2+\cc35\, (\delta K\,
\rme^{at})^4 \nonumber \\ \noalign{\vskip 2pt}
& \hspace{-50pt} \leq\ \cc34\, |\bfw|_{1,2}^2+\cc33\, \delta^2\,
\rme^{2at}+2\cc78\, (\delta\br K\, \rme^{at})^3+(\cc78^2+\cc35)\,
(\delta\br K\, \rme^{at})^4, \label{3.21}
\end{align}
where $\cc33$, $\cc34$ and $\cc35$ depend only on $\nu$.
Multiplying inequality inequality (\ref{4.6}) by
$\cn36\cc36:=(\cc49+\cc34)/\cc71$ and summing it with
(\ref{3.21}) and (\ref{4.5}), we obtain $\cn37$
\begin{displaymath}
\frac{\rmd}{\rmd t}\, \Bigl( |\bfw|_{1,2}^2+
|\bfz|_{1,2}^2+\cc36\, \|\bfw\|_2^2 \Bigr)\ \leq\ \cc33\,
\delta^2\, \rme^{2at}+2\cc78\, (\delta\br K\, \rme^{at})^3+
\cc37\, (\delta\br K\, \rme^{at})^4+\cc50\, (\delta\br K\,
\rme^{at})^6,
\end{displaymath}
where $\cc37:=\cc78^2+\cc35+\cc76\br\cc36$. Integrating this
inequality from $0$ to $t$ and summing with (\ref{3.17})
squared, we obtain
\begin{align*}
| & \bfw(t)|_{1,2}^2+|\bfz(t)|_{1,2}^2+\cc36\,
\|\bfw(t)\|_2^2+\|\bfz(t)\|_2^2 \\
& \leq\ \Bigl[\delta^2\, |\bfzeta+\rmi\br
\bfeta|_{1,2}^2+\frac{\cc33}{2a}\, \delta^2\,
\rme^{2at}+\frac{2\cc78}{3a}\, (\delta\br K\,
\rme^{at})^3+\frac{\cc37}{4a}\, (\delta\br K\, \rme^{at})^4+
\frac{\cc50}{6a}\, (\delta\br K\, \rme^{at})^6 \Bigr] \\
& \hspace{15pt} +\Bigl[ \delta^2\,
\|\bfzeta+\rmi\br\bfeta\|_2^2\ \rme^{2at}+
\frac{2\br\cc77}{K}\, \|\bfzeta+\rmi\br\bfeta\|_2\ (\delta
K\, \rme^{at})^3+\cc77^2\, (\delta K\, \rme^{at})^4 \Bigr] \\
&\leq\ \Bigl(1+\frac{\cc33}{2a}\Bigr)\,
\delta^2\, \rme^{2at}+\frac{2\br(3a\cc77+ \cc78)}{3a}\,
(\delta K\, \rme^{at})^3+
\Bigl(\frac{\cc37}{4a}+\cc77^2\Bigr)\,
(\delta K\, \rme^{at})^4+\frac{\cc50}{6a}\,
(\delta\br K\, \rme^{at})^6.
\end{align*}
We may assume without loss of generality that $\cc36>1$. Then
the left hand side is greater than or equal to
$\|\bfw(t)\|_{1,2}^2+\|\bfz(t)\|_{1,2}^2$. Thus, using the
inequality $\|\bfv(t)\|_{1,2}^2\leq 2\, \|\bfw(t)\|_{1,2}^2+2\,
\|\bfz(t)\|_{1,2}^2$, we get
\vspace{-8pt}
\begin{align*}
& \|\bfw(t)\|_{1,2}^2+\|\bfz(t)\|_{1,2}^2\ \leq\
\Bigl(1+\frac{\cc33}{2a}\Bigr)\, \delta^2\, \rme^{2at}+
\frac{2\br(3a\cc77+\cc78)}{3a}\, (\delta K\, \rme^{at})^3 \\
\noalign{\vskip 1pt}
& \hspace{115pt}
+\Bigl(\frac{\cc37}{4a}+\cc77^2\Bigr)\, (\delta K\,
\rme^{at})^4+\frac{\cc50}{6a}\, (\delta\br K\,
\rme^{at})^6.
\end{align*}
From this, we derive (\ref{3.26}) by standard manipulations.
The constants $\cc81$--$\cc84$ in (\ref{3.26}) depend only on
$a$, $\cc77$, $\cc78$, $\cc50$, $\cc33$ and $\cc37$.

\vspace{4pt} \noindent
{\bf Acknowledgement.} \ Part of this work was carried out when the first author was tenured with  the Eduard
\v{C}{e}ch Distinguished Professorship at the Mathematical Institute of the Czech Academy of Sciences in Prague. His work is also partially
supported by NSF Grant DMS-1614011 and   the
Mathematical Institute of the Czech Academy of Sciences (RVO
67985840). The second author also
acknowledges the support of the Grant Agency of the Czech
Republic (grant No.~17-01747S).

\noindent{\it Authors' addresses:}

\medskip \noindent
Giovanni Paolo Galdi, University of Pittsburgh, Department of
Mechanical Engineering, Pittsburgh, USA, e--mail:
galdi@pitt.edu

\medskip \noindent
Ji\v r\'\i\ Neustupa, Czech Academy of Sciences, Institute of
Mathematics, \v{Z}itn\'a 25, 115 67 Praha 1, Czech Republic,
e--mail: neustupa@math.cas.cz

\end{document}